\documentclass[12pt,reqno]{amsart}

    \makeatletter
    \newtheorem*{rep@theorem}{\rep@title}
    \newcommand{\newreptheorem}[2]{%
    \newenvironment{rep#1}[1]{%
     \def\rep@title{#2 \ref{##1}}%
     \begin{rep@theorem}}%
     {\end{rep@theorem}}}
    \makeatother

\usepackage[utf8]{inputenc} 
\usepackage[T1]{fontenc} 

\usepackage{amscd,graphicx}
\usepackage[all]{xy}
\usepackage{fullpage}
\usepackage{subsupscripts}

\usepackage{ifthen}

\newboolean{workmode}
\setboolean{workmode}{false}

\newboolean{sections}
\setboolean{sections}{true}

\usepackage[indent=0.5cm]{parskip}

	\AtBeginDocument{\addtocontents{toc}{\protect\setlength{\parskip}{0pt}}}
	\makeatletter
		\renewcommand*\l@subsection{\@tocline{2}{0pt}{2.75pc}{5pc}{}}
	\makeatother

	\usepackage[backend=bibtex,style=alphabetic,defernumbers=true,giveninits=false,doi=false,isbn=false,url=false,useprefix=true,maxnames=10,
	minnames=9,maxbibnames=99,maxalphanames=10,minalphanames=9,sorting = nyt,backref=false]{biblatex} 
	\bibliography{bibliography.bib} 

	\DeclareFieldFormat{labelalpha}{\mkbibbold{#1}}
	\DeclareFieldFormat{extraalpha}{\mkbibbold{\mknumalph{#1}}}
	\AtBeginBibliography{
		\DeclareFieldFormat{labelalpha}{#1}
		\DeclareFieldFormat{extraalpha}{\mknumalph{#1}}
	}

\usepackage{amsmath}

	\DeclareFontFamily{U}{mathx}{}
	\DeclareFontShape{U}{mathx}{m}{n}{<-> mathx10}{}
	\DeclareSymbolFont{mathx}{U}{mathx}{m}{n}
	\DeclareMathAccent{\widehat}{0}{mathx}{"70}
	\DeclareMathAccent{\widecheck}{0}{mathx}{"71}

    \DeclareFontFamily{U}{mathx}{\hyphenchar\font45}
    \DeclareFontShape{U}{mathx}{m}{n}{<-> mathx10}{}
    \DeclareSymbolFont{mathx}{U}{mathx}{m}{n}
    \DeclareMathAccent{\widebar}{0}{mathx}{"73}
	
\usepackage{amssymb}
\usepackage{amsthm}
\usepackage{thmtools} 
\usepackage{tikz-cd} 
\usepackage{rotating} 

\usepackage{xcolor}
\definecolor{jaw}{rgb}{0,.6,0}  
\definecolor{laura}{rgb}{.4, 0, .6}  
\definecolor{forestgreen}{rgb}{.2,.6,.2} 
\definecolor{darkturquoise}{RGB}{33,108,115} 

\usepackage[colorlinks=true,linkcolor={darkturquoise},citecolor={forestgreen},urlcolor={darkturquoise}]{hyperref}
\usepackage[capitalise,nameinlink,noabbrev]{cleveref}
	\crefname{subsection}{Subsection}{subsections}
	\crefdefaultlabelformat{#2\textbf{#1}#3} 

\usepackage{fontawesome}

\newcommand{\ifwork}[1]{\ifthenelse{\boolean{workmode}}{#1}{}}
\newcommand{\comment}[1]{}
\newcommand{\mute}[1]{}
\newcommand{\printname}[1]{}

\ifwork{
	\renewcommand{\comment}[1]{{\marginpar{*}\ \scriptsize{#1}\ }}
	\renewcommand{\printname}[1]
	{{\color{brown}{\makebox[0pt]{\hspace{-1in}\raisebox{8pt}{\tiny #1}}}}}
}

\newcommand{\labell}[1]{\label{#1} \printname{#1}}

    \newcommand{\zed}{\ensuremath{\mathbb Z}}

    \newcommand{\orbit}{\ensuremath{\mathcal{O}_G}}

		\newcommand{\calA}{{\mathcal{A}}}
		\newcommand{\calB}{{\mathcal{B}}}

		\newcommand{\calG}{{\mathcal{G}}}
		\newcommand{\calH}{{\mathcal{H}}}

		\newcommand{\calR}{{\mathcal{R}}}

		\newcommand{\calW}{{\mathcal{W}}}

		\newcommand{\CC}{{\mathbb{C}}}  
		\newcommand{\PP}{{\mathbb{P}}}  
		\newcommand{\QQ}{{\mathbb{Q}}}  
		\newcommand{\RR}{{\mathbb{R}}}  
		\renewcommand{\SS}{{\mathbb{S}}} 
		\newcommand{\ZZ}{{\mathbb{Z}}}  

		\newcommand{\Ab}{{\mathbf{Ab}}} 
        \newcommand{\act}{{\operatorname{act}}} 
        \newcommand{\Cat}{{\mathbf{Cat}}} 
        \newcommand{\Fun}{{\operatorname{Fun}}} 
		\newcommand{\id}{{\operatorname{id}}}  
		\newcommand{\inv}{{\operatorname{inv}}} 
		\newcommand{\LieGpd}{{\mathbf{LieGpd}}} 
		\newcommand{\mult}{{\operatorname{m}}} 
		\newcommand{\pr}{{\operatorname{pr}}} 
        \newcommand{\rel}{{\operatorname{rel}}} 
		\newcommand{\Rep}{{\operatorname{Rep}}} 
		\newcommand{\src}{{\operatorname{s}}} 
		\newcommand{\Stab}{{\operatorname{Stab}}} 
		\newcommand{\trg}{{\operatorname{t}}} 
		\newcommand{\unit}{{\operatorname{u}}} 
		
		\newcommand{\SO}{{\operatorname{SO}}}  
		\newcommand{\U}{{\operatorname{U}}}  
		
		\newcommand{\ftimes}[2]{{\lrsubscripts{\times}{#1}{#2}}} 
		\newcommand{\toto}{{~\rightrightarrows~}} 
		\newcommand{\ot}[1]{{\,\underset{\raisebox{1mm}{$\scriptstyle{#1}$}}{\overset{\simeq}{\scalebox{2}[1]{$\twoheadleftarrow$}}}\,}} 
		\newcommand{\subdwe}{{\,\overset{\simeq}{\scalebox{2}[1]{$\twoheadrightarrow$}}\,}} 
		\newcommand{\uto}[1]{{\,\underset{\raisebox{1mm}{$\scriptstyle{#1}$}}{\longrightarrow}\,}} 

		\usepackage{faktor}

		\makeatletter
		\DeclareRobustCommand*{\mfaktor}[3][]
		{
		   { \mathpalette{\mfaktor@impl@}{{#1}{#2}{#3}} }
		}
		\newcommand*{\mfaktor@impl@}[2]{\mfaktor@impl#1#2}
		\newcommand*{\mfaktor@impl}[4]{
		   \settoheight{\faktor@zaehlerhoehe}{\ensuremath{#1#2{#3}}}%
		   \settoheight{\faktor@nennerhoehe}{\ensuremath{#1#2{#4}}}%
		      \raisebox{-0.5\faktor@zaehlerhoehe}{\ensuremath{#1#2{#3}}}%
		      \mkern-4mu\diagdown\mkern-5mu%
		      \raisebox{0.5\faktor@nennerhoehe}{\ensuremath{#1#2{#4}}}%
		}
		\makeatother

	\newcommand{\ifsection}[2]{\ifthenelse{\boolean{sections}}{#1}{#2}} 

   		\theoremstyle{plain}
		\ifsection{
			\newtheorem{theorem}{Theorem}[section] 
		}
		{
			\newtheorem{theorem}{Theorem} 
		}
        \newreptheorem{theorem}{Theorem} 	
		\newtheorem*{maintheorem}{Main Theorem}
        \newtheorem*{theorem*}{Theorem}
		\newtheorem{proposition}[theorem]{Proposition}
		\newtheorem{corollary}[theorem]{Corollary}
		\newtheorem*{conjecture*}{Conjecture}
		
		\newtheorem{lemma}[theorem]{Lemma}
		\newtheorem*{lemma*}{Lemma}

		\theoremstyle{definition}
		\newtheorem{definitionx}[theorem]{Definition}
		\newtheorem{examplex}[theorem]{Example}

		\newtheorem{remarkx}[theorem]{Remark}

	\newenvironment{definition}
		{\pushQED{\qed}\definitionx}
		{\popQED\enddefinitionx}

	\newenvironment{example}
		{\pushQED{\qed}\examplex}
		{\popQED\endexamplex}

	\newenvironment{remark}
		{\pushQED{\qed}\remarkx}
		{\popQED\endremarkx} 

\usepackage{tikz}
\usetikzlibrary{decorations.markings,decorations.pathmorphing,arrows}

\tikzset{snake it/.style={decorate,decoration=snake}}

\newcommand{\firstmidarrow}{\tikz \draw[-{angle 90}] (0,0) -- +(.2,.1);}
\newcommand{\secondmidarrow}{\tikz \draw[-{angle 90}] (0,0) -- +(.2,.2);}

\newcommand{\ActGpd}{{\mathbf{ActGpd}}} 
\newcommand{\BI}{{\operatorname{BI}}} 
\newcommand{\eq}{{\operatorname{eq}}} 
\newcommand{\Coeff}{{\operatorname{Coeff}}} 

	\subjclass[2020]{Primary 22A22, 55N91; Secondary 58H05} 
	\keywords{action groupoids, Morita invariance, Bredon cohomology, Bredon-Illman cohomology} 

\begin{document}

	\author{Carla Farsi}
	\address{Department of Mathematics, University of Colorado at Boulder, Boulder, Colorado 80309, USA}
	\email{carla.farsi@colorado.edu}

    \author{Laura Scull}
	\address{Department of Mathematics, Fort Lewis College, Durango, Colorado 81301}
	\email{scull\_l@fortlewis.edu}

	\author{Jordan Watts}
	\address{Department of Mathematics, Central Michigan University, Pearce Hall 214, Mount Pleasant, Michigan 48859, USA}
	\email{jordan.watts@cmich.edu}
    
	\title{Twisted Bredon-Illman Cohomology is a Morita Invariant}
	\date{\today}

	\begin{abstract}
We show that the twisted Bredon-Illman cohomology defined by Mukherjee-Mukherjee applied to compact Lie group action groupoids is Morita-invariant.   This cohomology uses coefficient systems twisted over the discrete tom Dieck equivariant fundamental groupoid.    To show Morita  invariance, we use bibundles to transfer coefficient systems from one groupoid to another Morita equivalent one.   This generalizes results of Pronk-Scull and Juran on ordinary Bredon-Illman cohomology by removing both the finite isotropy condition and  restrictions on the coefficient systems.  
 
	\end{abstract}

	\maketitle

    \tableofcontents

	\thispagestyle{empty}

\section{Introduction}\labell{s:intro}

The study of the symmetries of topological spaces is an extensive field of study with wide applications across many fields. Spaces with group actions form the basis for equivariant topology, and many topological invariants have been adapted to their equivariant version to capture the interplay of the action with the topological features of the space.   
 One of the most flexible of these equivariant  invariants is Bredon-Illman cohomology (often just called Bredon cohomology), defined in \cite{bredon:eqvt,illman}.  This theory incorporates information about the fixed sets of the group action and builds on work that treats equivariant spaces as diagrams of fixed sets indexed by the orbit category $\orbit$ of the acting group.
Given a topological group $G$, the orbit category has objects defined by the canonical orbit types $G/H$ for $H$ a closed subgroup, with equivariant maps between them.   Then the fixed sets of any $G$-space with their induced $G$-actions and inclusions form a diagram indexed by $\orbit$. 
 Bredon-Illman cohomology makes use of this fixed set diagram to give    more detailed information about the equivariant structure than its sibling equivariant cohomology theory, Borel cohomology; see \cite[Section IV.2]{may} for a comparison.      Bredon cohomology has been used as a basis for developing equivariant obstruction theory, orientation theory, and covering spaces (see, for instance, \cite{CMW,CW}). 

Equivariant topology often starts by fixing a group and studying all spaces which have an action of that group.   There is another way of looking at information induced by symmetries which views the $G$-space as a quotient space $X/G$ endowed with extra structure,   retaining information about the isotropy group and local geometric structure;  this is the approach taken when studying orbifolds (and more generally stacks).  Orbifolds were originally defined as a generalization of manifolds, created using charts consisting of an open subset of $\mathbb{R}^n$ with the action of a finite group, which can be a different group on each chart, together with maps between them.   In this world, it is natural to consider multiple different groups acting within the same geometric structure. Moreover, smooth orbifolds can  be represented by an  action of a compact Lie group on a manifold \cite{satake}.   This encoding is not unique: it is possible to have two different Lie groups acting on their own manifolds representing the same orbifold, in which the quotients are homeomorphic and have the same isotropy groups and local geometry.   

In order to give a  framework for moving between the actions of different groups and describing this non-uniqueness, we encode a $G$-space via  an action groupoid:    if $G$ acts on a space $X$, then the action groupoid $G \ltimes X$ has object space $X$ and arrows $G \times X$, where $(g, x)$ represents an arrow  $x \to gx$.   Then two group actions represent the same stack when their associated action groupoids are what is now called Morita equivalent (see \cref{d:biprincipal}).  In fact, for stacks represented by these action groupoids, this can be taken as the definition of the stack.  Furthermore, if the action groupoids are proper and \'etale, then these stacks are orbifolds.

The idea of Morita invariance is named after Morita \cite{morita}, and it originally applied to rings: two rings are Morita equivalent if their categories of modules are equivalent; this can be captured with a bimodule \cite{anderson-fuller}.  Later, this was applied to $C^*$-algebras via the celebrated Muhly-Renault-Williams Theorem \cite{MRW}: given two locally compact groupoids with Haar systems and a biprincipal bibundle between them (see \cref{d:biprincipal}), the corresponding $C^*$-algebras are ``Morita equivalent'' in the sense that they admit a special bimodule between them, built from the biprincipal bibundle.  Eventually, such groupoids became referred to as Morita equivalent as well.

The above theory holds for both topological orbifolds as well as the smooth versions, represented by topological or Lie groupoids, respectively.  In this paper, we focus only on Lie groupoids as the smooth theory allows one to use local linearization results, such as the Slice Theorem and a result of Bierstone concerning equivariant bundles, both of which we make use of in the technical heart of the paper (namely, the proof of \cref{l:lifting}).  However, we do not require finite isotropy and so work in a more general context than just orbifolds. 

Morita equivalence between Lie groupoids is now understood in a more categorical context.   Working in the 2-category of Lie groupoids, these equivalences are generated by a smooth functor of Lie groupoids which induces an equivalence of categories (called a ``weak equivalence'' in this paper; see \cref{d:ess equiv}).  Unlike abstract categories, a smooth inverse to a weak equivalence may not exist.   Thus we need to  formally invert these morphisms in a way analogous to localization in the theory of rings:  localizing the 2-category at the class of all weak equivalences results in a bicategory  in which  Morita equivalent groupoids are now isomorphic.  There are many different recipes for this localization process \cites{haefliger, HS, pronk,lerman:orbifolds,, roberts:anafunctors, PS:bicateg,,FSW:action}.  We take the approach of bibundles in this paper: see \cref{d:gpd-action} and \cref{d:biprincipal} for the definition of a bibundle, and the manifestation of Morita equivalence as a biprincipal bibundle.  Orbifolds, and more generally stacks, can now be thought of as objects in this localized category, independent of the representative action groupoid.

In this framework, any equivariant theory can be applied to a particular representing $G$-manifold, encoded as an action groupoid, but it will only be well-defined for orbifolds in the localized category if it is invariant under Morita equivalences. 
Some equivariant cohomology theories, such as   Borel cohomology  and  equivariant K-theory, are known to be Morita invariant; see more below in the list of related results on Morita invariance.

The non-uniqueness of a representative action groupoid  presents a challenge for defining a Morita invariant Bredon-Illman cohomology: while the isotropy groups are Morita invariant, the organization into fixed sets indexed by the orbit category $\orbit$ is not.   Specifically, it is possible to have two connected components belonging to the same fixed set in one representative groupoid and different fixed sets in another.   Thus they show up in the same position in the $\orbit$ diagram for one representative and in different ones in another.     This makes it  possible to define Bredon-Illman cohomology  which assigns distinct coefficients to the two components in one representative but not another.  Bredon-Illman cohomology theory has been extended to representable orbispaces, a generalization of orbifolds,  in Pronk-Scull \cite{PS:translation}.   In this work, the authors consider actions of compact Lie groups with finite isotropy.  Many orbifolds can be represented in this way; see \cite{henriques-metzler,pardon}.  In order to account for the difficulties above,  Pronk-Scull impose conditions on the coefficient systems (called ``orbifold coefficient systems'').  For these restricted coefficients,   Bredon-Illman cohomology is Morita-invariant, thus giving an orbifold invariant.  They also give an example of a pair of Morita equivalent action groupoids in which the mismatch in fixed sets occurs:  \cite[Example 5.3]{PS:translation}  gives two Morita equivalent groupoids which represent the same orbifold.    In one representative, a single fixed set consists of two components,  while in the other these two components have distinct (but still isomorphic) isotropy groups;  see also \cref{x:ac}.   Therefore, the two representative groupoids yield different Bredon-Illman cohomologies using coefficient systems that are not of the orbifold type.  
 
Thus we see that the orbit category $\orbit$ may not sufficiently distinguish between the components of a fixed set.  This led us to considering the so-called disconnected Bredon cohomology theory, developed for finite group actions by Fine \cite{fine}; see also \cite{golasinski:disconnected}.  This version of Bredon-Illman cohomology specifically allows coefficient systems which treat different connected components of a fixed set differently.  Our attempts to generalize this to compact group actions and prove Morita invariance failed: we require a ``pushforward coefficient system'' (see \cref{p:pushforward coeff system}), and using the set-up of disconnected Bredon cohomology failed to produce a well-defined such functor.  In attempting to remedy this problem, we  were lead to consider a twisted Bredon-Illman cohomology theory over the discrete tom Dieck fundamental groupoid. This approach was further encouraged by the fact that  Pronk-Scull had already showed that the discrete tom Dieck fundamental groupoid is a Morita invariant of representable orbifolds \cite{PS:fund-gpd}.  

Bredon-Illman cohomology with coefficient systems defined by diagrams indexed by the discrete tom Dieck fundamental groupoid has been developed by Mukherjee-Mukherjee \cite{MuMu}, called Bredon-Illman cohomology with local coefficients there.   We will call this theory ``twisted'' Bredon-Illman cohomology.  The discrete fundamental groupoid does distinguish between different components of a fixed set, and so coefficients defined using it can assign different coefficients even if they have the same isotropy subgroup.

In this paper, we prove the following (it appears as \cref{t:bredon-illman} in the paper):
\begin{maintheorem}[Main Theorem]
    If $G\ltimes X$ and $H\ltimes Y$ are Morita equivalent Lie group action groupoids with $G$ and $H$ compact, and $\calA$ is a coefficient system on the discrete tom Dieck fundamental groupoid $\Pi(H\ltimes Y)$, then we can transfer the coefficient system $\calA$ to a coefficient system $\calB$ on $\Pi(G \ltimes X)$ via a biprincipal bibundle so that there is an isomorphism between the twisted Bredon-Illman cohomologies $H^\bullet_\BI(G\ltimes X;\calB)\cong H^\bullet_\BI(H\ltimes Y;\calA)$.
\end{maintheorem}
This means that Bredon-Illman cohomology is defined (up to isomorphism) on stacks represented by action groupoids of compact Lie group actions.  This setting is more general  than orbifolds, but in particular implies that  twisted Bredon-Illman cohomology is a  Morita invariant of any orbifold, with no conditions on the coefficient systems, thus generalizing \cite{PS:translation}.  Moreover, we obtain an explanation for the earlier restriction on coefficient systems:  we may have an orbifold in which an ordinary Bredon-Illman cohomology for one representative corresponds to a twisted Bredon-Illman cohomology for a different representative.   We revisit \cite[Example 5.3]{PS:translation} in \cref{x:ac} and \cref{x:c3}, and show that the two presentations of the orbifold yield isomorphic twisted Bredon-Illman cohomologies. Finally, we reframe our results into a statement that twisted Bredon-Illman cohomology is a functor from the Grothendieck construction pairing action groupoids with their coefficient systems  (\cref{c:bredon-illman}).

The idea of twisting coefficients over the fundamental groupoid appears also in work by  Moerdijk–Svensson for actions of discrete groups \cites{moerdijk-svensson}.    Bredon-Illman cohomology can also be approached via representability. For discrete groups, representability of twisted Bredon cohomology by a naive spectrum was established
by Mukherjee-Sen in \cite{MuSen}, with further results developed by  Basu-Sen \cite{BaSen}.  The idea of extending to twisted coefficients also shows up in stable equivariant theory, particularly in the work of Costenoble-Waner \cite{CW} who extend the usual $RO(G)$-graded cohomology theories to those graded  by $RO(\Pi(B))$, representations on the fundamental groupoid;   see also Beaudry-\emph{et al} for an overview \cite{Beaudry}.  In this  work we consider unstable cohomology only, and do not consider the stable category or work with spectra.

In recent years,the language of $\infty$-categories and $\infty$-groupoids, objects together with their higher-dimensional morphisms and homotopies, has become prominent  in understanding localization \cite{lurie}.  Orbifolds may be  viewed as higher-categorical objects rather than as ordinary spaces, as a suitable class of geometric stacks, which in turn form an $\infty$-topos.     An illustration of this perspective is provided by \cite{dai}, where classical results concerning global presentations of orbispaces are revisited using $\infty$-categorical methods;  see also \cite{carchedi} for further applications to orbifolds and Deligne-Mumford stacks.
From this perspective, Morita equivalence may be regarded as becoming an equivalence after passing to the associated homotopy-theoretic or $\infty$-categorical object.

In the present paper we do not pursue the language of $\infty$-categories and $\infty$-groupoids.  Nevertheless, it seems plausible that our results, especially related to the functoriality and the localization of weak equivalences, admit a more conceptual interpretation within a higher-categorical framework. In particular, the functor of \cref{c:bredon-illman} is only defined up to isomorphism on the bicategory of action groupoids localized at equivariant weak equivalences, suggesting that to fully express the functoriality of twisted Bredon cohomology will require the use of the higher categorical framework.
Elmendorf’s Theorem \cite{elmendorf} identifies the homotopy theory of $G$-spaces with
the homotopy theory of presheaves on the orbit category.   Bredon-Illman
cohomology may therefore be interpreted in terms of homotopy-theoretic or $\infty$-categorical
structures associated with the orbit category. Very recently, a sheaf-theoretic
refinement of Bredon-Illman cohomology for finite groups was introduced in \cite{ArMuNi}. This theory may be viewed as a sheaf-theoretic enhancement of the higher-categorical approach to
Bredon-Illman cohomology.

 Our main theorem joins other Morita-invariant related results already in the literature,  including the following, ordered from most to least relevant:   
    \begin{itemize}
        \item We have already mentioned \cite{PS:translation} above.
        \item Juran in \cite{juran} uses orthogonal spectra to study stable equivariant cohomology theories of orbifolds, following ideas of  Schwede \cite{Schwede1, Schwede2}.  This setting allows for all possible compact Lie groups actions, and gives a framework for comparing Morita equivalent action groupoid representatives. 
  Juran shows that if a coefficient system extends to a Mackey functor, thus defining a cohomology theory in the stable world,     it is Morita invariant.  
        \item Farsi-Scull-Watts show that the Bredon-Illman cohomology of an action of a transitive topological groupoid is isomorphic to the Bredon-Illman cohomology of an action of the stabilizer group of the groupoid in \cite{FSW:transitive}; the two groupoids involved are (topologically) Morita equivalent. 
      
        \item It is known that the rational equivariant K-theory of a global quotient orbifold is isomorphic to the ordinary Bredon-Illman cohomology with values in the representation ring coefficient system $\calR$ (similar to that appearing in \cref{c:rep coeff sys} and \cref{x:d4}) tensored with $\QQ$ or $\CC$ \cite{moerdijk:orbifolds,luck-oliver}, and this in turn is independent of the groupoid representing the orbifold \cite{AdemRuanTw}:  if two action groupoids $G\ltimes M$ and $H\ltimes X$ are 
Morita equivalent, then their equivariant topological K-theory is the same  \cite{Blackadar}.    In particular, quoting \cite[page 540]{AdemRuanTw}:
        \begin{quote}
            ``Suppose that $X = M/G$ is a quotient orbifold. Using equivariant K-theory, we will show that the Bredon cohomology $H_G^*(M;\calR_\QQ)$ is independent of the
            presentation $M/G$, and canonically associated with the orbifold $X$ itself. A direct proof with more general coefficients would be of some interest.'' 
        \end{quote}
        We achieve this more direct proof in \cref{c:rep coeff sys} without using K-theory, and also in more generality since it holds for general compact Lie group actions.  Particular cases of this appear in \cite{PS:translation} and \cite{juran} for topological global quotient orbifolds.

    \item For a sufficiently nice $G$-space $X$ with a coefficient system $\calA$ defined on $\orbit$,  Honkasalo in \cite{honkasalo} defines an equivariant Alexander-Spanier cohomology on $X$ with respect to $\calA$, and a sheaf $S(\calA)$ on the orbit space $X/G$ whose sheaf cohomology is isomorphic to the equivariant Alexander-Spanier cohomology with coefficients from $\calA$ \cite[Theorem 5.5]{honkasalo}.  Further, Honkasalo shows under certain conditions that the equivariant Alexander-Spanier cohomology is isomorphic to the ordinary Bredon-Illman cohomology.  Since the orbit space is a Morita invariant, this result strongly hints at a Morita-invariance result for Bredon-Illman cohomology.
        \item In \cite{moerdijk-svensson}, a similar category to that developed in \cite{MuMu} is defined for a $G$-space when  $G$ is discrete for the purposes of computing the Serre spectral sequence in the equivariant setting of a $G$-fibration.  The resulting cohomology is equivalent to that of \cite{MuMu} (and this work)  when $G$ is a discrete group \cite{mukherjee-pandey}, and thus the twisted Bredon-Illman cohomology of \cite{MuMu} generalizes that of \cite{moerdijk-svensson}.   This work does not directly consider Morita invariance, but our result shows that in the case when $G$ is finite, the cohomology of \cite{moerdijk-svensson} is a Morita invariant.        
    \end{itemize}

The paper is organized as follows.  In \cref{s:act gpds}, we introduce the preliminaries required for the paper; namely, we define Lie groupoids, bibundles, and introduce the main examples we will use to illustrate the theory throughout the paper.  In \cref{s:fund gpd}, we review the discrete tom Dieck fundamental groupoid and establish its Morita invariance for action groupoids of compact Lie group actions \cref{t:fund gpd morita}. In \cref{s:coeff systems}, we study in detail the equivariant interactions present in a bibundle, and use these to define pullback and pushforward coefficient systems.  The pushforward coefficient system is crucial to the paper, and depends on choices.  Different choices, however, yield coefficient systems that differ by a natural isomorphism \cref{p:pushforward coeff system}, and due to this the twisted Bredon-Illman cohomologies in the end will be isomorphic.  Finally, in \cref{s:bredon-illman}, we review twisted Bredon-Illman cohomology, and prove our main result, ending the paper with a couple of applications and example computations.

\subsection*{Acknowledgements:}  The authors wish to thank Dorette Pronk for many discussions leading to this paper.  The first author's research was partly supported by the Simon's collaboration grant MPS-TSM-00007731.  The second and  third author thank the University of Colorado Boulder for their hospitality during a crucial part of this work.  The third author acknowledges the support of Central Michigan University via a research grant allowing the author to travel to Boulder.   Finally, the authors thank the anonymous referee and the editor for their helpful suggestions.

\section{Action Groupoids \& Bibundles}\labell{s:act gpds}

In this section, we outline the basic setting for the rest of the paper.  Our goal is to show Morita invariance of twisted Bredon-Illman cohomology, so this section focuses on the idea of Morita equivalence and the approach that we will be taking to it.  We work in the context of Lie groupoids, and we introduce the weak equivalences that generate Morita equivalence.  We then review how we can localize the category of Lie groupoids at weak equivalences using bibundles following \cite{lerman:orbifolds,MM:sheaves}, setting up the categorical framework for our work. 

Once we have sketched out this general framework, we introduce our objects of interest, the  action groupoids with equivariant maps between them.  These groupoids come specifically from the action of a compact Lie groupoid on a manifold.     We discuss localization in the context of  these action groupoids and give a concrete description of what bibundles look like in this context.    We illustrate our definitions with examples which will be used going forward throughout the rest of the paper.  

Throughout this paper, `manifold' will mean a smooth manifold without boundary, all actions are smooth, all groups are compact Lie and and all subgroups are closed. 

We begin with the 2-category of Lie groupoids.  

    \begin{definition}[Lie Groupoid]\labell{d:lie gpd}
        A \textbf{Lie groupoid} $\calG=(\calG_1\toto\calG_0)$ is a small category in which the space of objects $\calG_0$ and arrows $\calG_1$ (all of which are invertible) are smooth manifolds, and the source map $\src$, target map $\trg$, unit map $\unit$, multiplication map $\mult$, and inversion map $\inv$ are all smooth, with source and target maps submersive.

        A \textbf{smooth functor} $\varphi\colon\calG\to\calH$ between Lie groupoids is a functor whose defining maps on objects and arrows are both smooth.   (It suffices to assume the functor is smooth on arrows here, as smoothness on objects follows.)

        Given smooth functors $\varphi,\psi\colon\calG\to\calH$, a \textbf{smooth natural transformation} is a natural transformation $S\colon\varphi\Rightarrow\psi$ that is defined by  a smooth map $\calG_0\to\calH_1$.

        Together, Lie groupoids, smooth functors, and smooth natural transformations form a (strict) $2$-category $\LieGpd$.
    \end{definition}

We consider Lie groupoids up to Morita equivalence, which means we want to invert a class of smooth functors known as weak equivalences. 

\begin{definition}[Weak Equivalence]\labell{d:ess equiv}
		A functor $\varphi\colon\calG\to\calH$ is a \textbf{weak equivalence} (sometimes  called an \textbf{equivalence} or an \textbf{essential equivalence} in the literature) if it satisfies the following two conditions: 
		\begin{enumerate}
			\item \textbf{Smooth Essentially Surjective:} The induced  map $(\calG_0)\ftimes{\varphi_0}{t}\calH_1\to\calH_0\colon (x,h)\mapsto \src(h)$ is a surjective submersion.
			\item  \textbf{Smooth Fully Faithful:} The induced map $\calG_1\to(\calG_0^2)\ftimes{\varphi_0^2}{(s,t)}\calH_1\colon g\mapsto (\src(g),\trg(g),\varphi(g))$ is a diffeomorphism.\qedhere
		\end{enumerate}
\end{definition}

We want to pass to a category which localizes these weak equivalences.  There are a number of ways to approach this in the literature, including bicategories of fractions of \cite{pronk}, anafunctors \cite{roberts:anafunctors} and bibundles of \cite{lerman:orbifolds,MM:sheaves}.   In this work, we will use the bibundle approach.   To define these, we need the following.  
    
    \begin{definition}(Groupoid Actions, Principal Bundles, and Bibundles)\labell{d:gpd-action}
        Let $\calG$ be a Lie groupoid and $Z$ a smooth manifold. A \textbf{left groupoid action} of $\calG$ on $Z$ is a Lie groupoid $\calG\ltimes Z$ with arrow space $(\calG_1)\ftimes{\src}{a}Z$ and object space $Z$, where $a\colon Z\to\calG_0$ is the smooth \textbf{anchor map}.  The source map is the projection onto $Z$ and the target map is the smooth action map $\act\colon(g,z)\mapsto g\cdot z$, which is required to satisfy:
        \begin{enumerate}
            \item $\unit_{a(z)}\cdot z=z$ for all $z\in Z$,
            \item $a(g\cdot z)=\trg(g)$ for all $(g,z)\in(\calG\ltimes Z)_1$, and
            \item $g_1\cdot(g_2\cdot z)=(g_1g_2)\cdot z$ for composable $g_1$ and $g_2$ such that $(g_1g_2,z)\in(\calG\ltimes Z)_1$.
        \end{enumerate}

        A \textbf{left principal $\calG$-bundle} $\rho\colon Z\to Y$ is a left $\calG$-action on $Z$ such that
        \begin{enumerate}
            \item $\rho$ is a surjective submersion,
            \item $\rho$ is \textbf{$\calG$-invariant}: $\rho(g\cdot z)=\rho(z)$ for all $(g,z)\in(\calG\ltimes Z)_1$, and
            \item the map $$A\colon(\calG\ltimes Z)_1\to Z\ftimes{\rho}{\rho}Z\colon(g,z)\mapsto(z,g\cdot z)$$ is a diffeomorphism.
        \end{enumerate}
A similar list of conditions defines a {right groupoid action} and a  \textbf{right principal $\calH$-bundle}. 

        A \textbf{bibundle} $Z$ from $\calG$ to $\calH$, denoted $\calG\ot{\lambda}Z\uto{\rho}\calH$, is a right principal $\calH$-bundle $\lambda\colon Z\to\calG_0$ with anchor map $\rho\colon Z\to\calH_0$ such that $\rho$ is a $\calG$-action with anchor $\lambda$, the actions of $\calG$ and $\calH$ commute, and $\rho$ is $\calG$-invariant.

        Given two bibundles $Z$ and $Z'$ between $\calG$ and $\calH$, a \textbf{biequivariant diffeomorphism} $F\colon Z\to Z'$ is a diffeomorphism such that $F(g\cdot z\cdot h)=g\cdot F(z)\cdot h$ for all $(g,z,h)\in(\calG_1)\ftimes{\src}{\lambda} Z\ftimes{\rho}{\trg}(\calH_1)$.
    \end{definition}

 We can use these bibundles as morphisms between Lie groupoids in place of smooth functors.  

    \begin{definition}[Bibundle from Functor]\labell{d:functor-bibundle}
         Given a smooth functor $\varphi\colon\calG\to\calH$, there is a corresponding bibundle $P_\varphi:=\left(\calG\ot{\lambda}(\calG_0)\ftimes{\varphi}{\trg}\calH_1\uto{\rho}\calH\right)$ from $\calG$ to $\calH$ with left $\calG$-action given by $g\cdot(x,h)=(\trg(g),\varphi(g)h)$ with anchor map $\lambda(x,h)=x$, and right $\calH$-action given by $(x,h)\cdot h'=(x,hh')$ with anchor map $\rho\colon(x,h)\mapsto\src(h)$. 
    \end{definition}
     
 Under this correspondence our weak equivalences wind up as biprincipal bibundles.   
 
    \begin{definition}[Biprincipal Bibundle]\labell{d:biprincipal}
         A bibundle as in \cref{d:gpd-action} is \textbf{biprincipal} if both $\rho$ and $\lambda$ are principal bundles with respect to the $\calG$- and $\calH$-actions.
         
         If there exists a biprincipal bibundle between $\calG$ and $\calH$, then way say that $\calG$ and $\calH$ are \textbf{Morita equivalent}. 
    \end{definition} 
 
    \begin{proposition}[Weak Equivalences \& Biprincipality]\labell{p:functor bibundle}
         If $\varphi$ is a weak equivalence, then the bibundle of \cref{d:functor-bibundle} is biprincipal.  
    \end{proposition}

    We can form a bicategory with objects given by Lie groupoids, arrows given by  bibundles and 2-cells defined by biequivariant diffeomorphisms between bibundles. In this bicategory,  we can weakly invert a biprincipal bibundle by simply reversing their order.  
 
Within this context, we will focus our attention on the \textbf{action groupoids}.  

    \begin{definition}[Action Groupoids]\labell{d:gp action}
        Let $G$ be a Lie group acting on a manifold $X$ on the left.  The corresponding \textbf{action groupoid} $G\ltimes X$ has arrow space $G\times X$ and object space $X$; the source map is the projection onto $X$ and the target map is the action map $(g,x)\mapsto g\cdot x$.  Similarly, a right action of a Lie group $H$ on $X$ induces a Lie groupoid $X\rtimes H$.
    \end{definition}
 
Below we give some examples of action groupoids.  These will be referenced and serve as running examples throughout the remainder of the paper. 

    \begin{example}[{\bf Groupoid A}]\labell{x:a1}
        Define the action groupoid  $G\ltimes X:=\ZZ/2\ltimes\SS^1$, where $\ZZ/2$ acts on the circle  by reflection across the $y$-axis in $\CC\supset\SS^1$. 
    \end{example}

    \begin{example}[{\bf Groupoid B}]\labell{x:b1}
        Let $Y:=\U(1)\times_{\ZZ/2} \SS^1$ be the quotient of $\U(1)\times \SS^1$ by the diagonal action of $\ZZ/2$, and define the $H:=\U(1)$-action on $Y$ given by $\U(1)\times_{\ZZ/2}\SS^1$, with action of $\U(1)$ given by multiplication on the left. We consider the action groupoid $H \ltimes Y$.  The space $Y$ is a Klein bottle which can be viewed as a compact cylinder with the $\U(1)$-action spinning it on its axis, and both ends replaced with copies of $\RR\PP^1$ (\emph{i.e.}\ cross-caps) on which $\U(1)$ acts ineffectively with representation kernel equal to $\ZZ/2\cong\{\pm 1\}\leq\U(1)$.  
    \end{example}

    \begin{example}[{\bf Groupoid C}]\labell{x:c1}
        Define the groupoid  $D_4\ltimes\SS^1$, the action of dihedral group of order $4$ (equal to $\ZZ/2\times\ZZ/2$) acting on $\SS^1$ by reflections: $(-1,1)$ reflects across the $y$-axis, and $(1,-1)$ across the $x$-axis. The former fixes $N$ and $S$, the north and south poles of $\SS^1$, whereas the latter fixes $E$ and $W$, the east and west poles of $\SS^1$.  
    \end{example}

    \begin{example}[{\bf Groupoid D}]\labell{x:d1}
        Define   $\SO(n) \ltimes\SS^n$ where the action of $\SO(n)$ on $\SS^n$ is defined by rotating $\SS^n$ about the axis through its north and south poles. 
    \end{example}

We now look at equivariant maps between action groupoids.  

   \begin{definition}[Equivariant Functor]\labell{x:eqvar functor}
        Given Lie groups $G$ and $H$ acting on manifolds $X$ and $Y$ respectively, action groupoids $G\ltimes X$ and $H\ltimes Y$,  a smooth functor $\varphi\colon G\ltimes X\to H\ltimes Y$ is an \textbf{equivariant functor} if $\varphi(g,x)=\left(\widetilde{\varphi}(g),\varphi(x)\right)$ for some Lie group homomorphism $\widetilde{\varphi}\colon G\to H$.  In particular, this encodes all \textbf{$\widetilde{\varphi}$-equivariant maps}; that is, maps $\varphi\colon X\to Y$ such that $\varphi(g\cdot x)=\widetilde{\varphi}(g)\cdot\varphi(x)$.
    \end{definition}

 Lie group action groupoids, equivariant functors, and smooth natural transformations between these give us the  sub-$2$-category $\ActGpd$ of $\LieGpd$.

 Now we want to look at what the localization of weak equivalences looks like for  $\ActGpd$.   If our groupoids are action groupoids, then a bibundle between them has a recognizable form. To describe this, we use the following. 
    \begin{definition}[Equivariant Principal Bundle: Bierstone, \cite{bierstone:eqvt-bdle}]\labell{d:eqvt princ bdle}
        Given Lie groups $G$ and $H$, a principal $H$-bundle $\lambda\colon Z\to X$ is \textbf{$G$-equivariant} if $Z$ and $X$ are $G$-manifolds, $\lambda$ is $G$-equivariant, and the $G$- and $H$-actions on $Z$ commute.  Such a bundle is \textbf{$G$-locally trivial} if for each $x\in X$ there is a $\Stab(x)$-invariant neighborhood $U$ of $x$ and a $\Stab(x)$-equivariant bundle diffeomorphism $\psi\colon\lambda^{-1}(U)\to U\times\lambda^{-1}(x)$.
    \end{definition}

    Some authors allow the $G$- and $H$-actions on $Z$ to interact non-trivially in \cref{d:eqvt princ bdle}, inducing an action of the semidirect product $G\ltimes H$ on  $Z$.  We will not require this much generality.  An alternative definition uses slices (modeled on normal spaces to orbits) to define an equivariant bundle (see, for instance, Lashof \cite{lashof}), but this is equivalent to Bierstone's definition above \cite[Lemmas 1.1,1.2]{lashof}\footnote{Replace continuity with smoothness where appropriate.}.  

    An important result concerning equivariant principal bundles is: 

    \begin{theorem}[Local Triviality of Equivariant Bundles (Bierstone), {\cite[Theorem 4.2]{bierstone:eqvt-bdle}}]\labell{t:bierstone}
        If $G$ and $H$ are compact Lie groups and $\lambda\colon Z\to X$ is a $G$-equivariant principal $H$-bundle, then $\lambda$ is $G$-locally trivial.
    \end{theorem}

These equivariant principal bundles are exactly what we need to describe bibundles between action groupoids. 

    \begin{proposition}[Bibundles Between Action Groupoids, {\cite[Proposition 33]{FSW:action}}]\labell{p:bibundle form}
        If $G\ltimes X$ and $Y\rtimes H$ are left and right Lie group action groupoids, a bibundle $Z$ from $G\ltimes X$ to $Y\rtimes H$ is exactly a  $G$-equivariant principal $H$-bundle $\lambda\colon Z\to X$ as in  \cref{d:eqvt princ bdle} equipped with a $(\pr_2\colon G\times H\to H)$-equivariant functor $\rho\colon Z\to Y$.      
    \end{proposition}

    \begin{remark}\labell{r:bibundle}
        In the rest of the paper,  in order to consistently define composition in the tom Dieck fundamental groupoid in \cref{s:fund gpd}, we will treat the $H$-action on a bibundle $Z$ as in \cref{p:bibundle form} as a left-action; thus, $G\times H$ acts on $Z$ by $(g,h)\cdot z=gzh^{-1}$.
    \end{remark}

    \begin{corollary}\labell{c:bibundle form}
        If $G \ltimes X$ and $H \ltimes Y$ are two Morita equivalent action groupoids, then the biprincipal bibundle $G\ltimes X\ot{\lambda}Z\underset{\raisebox{1mm}{$\scriptstyle\rho$}}{\subdwe} H\ltimes Y$ between them is both a $G$-equivariant principal $H$-bundle $\lambda\colon Z \to X$ and an $H$-equivariant principal $G$-bundle $\rho\colon  Z \to Y$.  
    \end{corollary}
    
We form a bicategory $\ActGpd[\calW^{-1}_{\eq}]$ with objects given by  action groupoids of Lie group actions, arrows given by bibundles (which are of the form of \cref{p:bibundle form}), and 2-cells given by  biequivariant diffeomorphisms.   Then  $\ActGpd[\calW^{-1}_{\eq}]$ is  the localization of $\ActGpd$ at equivariant weak equivalences $\calW_\eq$ \cite[Theorem 38]{FSW:action}. This, in turn, is equivalent to the full sub-bicategory of action groupoids localized at all weak equivalences between them.  Similar statements hold for the restriction to action groupoids of Lie group actions satisfying certain properties, including compactness.  See \cite{FSW:action} for details, in which the authors use so-called anafunctors instead of bibundles.

We will work in the setting of $\ActGpd[\calW^{-1}_{\eq}]$ for the rest of this paper.   This allows us to recognize a Morita equivalence between action groupoids either by a functor which is an equivariant weak equivalence, or by a biprincipal (equivariant) bibundle corresponding to it as in   \cref{d:functor-bibundle}.  The proofs of the weak equivalence claims below are straightforward, hence omitted.

    \begin{example} \labell{x:ab} There is an equivariant weak equivalence from Groupoid A of \cref{x:a1} to Groupoid B of \cref{x:b1} defined by the functor $\varphi\colon \ZZ/2\ltimes\SS^1\to\U(1)\ltimes\U(1)\times_{\ZZ/2}\SS^1$ sending $(\pm 1,x)$ to $(\pm 1,[1,x])$.  The corresponding bibundle $P_\varphi$ to this functor is $K\ltimes Z$ where $K=\U(1)\times\ZZ/2$ and $Z$ is the torus $\U(1)\times\SS^1$, on which $K$ acts by $$(e^{i\theta},\pm 1)\cdot(e^{i\alpha},e^{i\beta}):=(\pm e^{i(\alpha+\theta)},e^{\pm i\beta}).\eqno\qedhere$$
 This is a particular example of one of the two fundamental types of equivalences in \cite{PS:translation}.
    \end{example}

    \begin{example} \labell{x:ac}
        There is an equivariant weak equivalence from Groupoid C of \cref{x:c1} to Groupoid A.   The map $z\mapsto iz^2$ from $\SS^1$ to itself induces an equivariant weak equivalence $\psi$ from Groupoid C, $D_4\ltimes\SS^1$, to Groupoid A, $\ZZ/2\ltimes\SS^1$; the corresponding group homomorphism is the quotient homomorphism $\widetilde{\psi}\colon D_4\to D_4/\langle (-1,-1)\rangle\cong\ZZ/2$.  Note that while $\ZZ/2\times 1$ and $1\times\ZZ/2$ are isomorphic as groups, they are not even conjugate as subgroups of $D_4$; however, they are both sent by $\widetilde{\psi}$ to the stabilizers of $N$ and $S$ in $\ZZ/2\ltimes\SS^1$.  This example is \cite[Example 5.3]{PS:translation}, and is a simple example of how two Morita equivalent orbifolds can have very distinct diagrams of fixed point sets, and hence, non-isomorphic (ordinary) Bredon-Illman cohomology.    We will revisit this example later. 
    \end{example}

Thus the action groupoids of Examples A, B, and C are all Morita equivalent and will all be isomorphic in our localized category $\ActGpd[\calW^{-1}_{\eq}]$.  

    \begin{example}\labell{x:d2}  The action groupoids of Groupoid D,  \cref{x:d1} are \emph{not} Morita equivalent for distinct $n$. 
       The action of  $\SO(n)$ on $\SS^n$ by rotating $\SS^n$ about the axis through its north and south poles has   orbit spaces which are all homeomorphic (in fact, diffeomorphic\footnote{Here, we mean diffeomorphic when equipped with the quotient Sikorski differential structure \cite{sniatycki}, which makes them diffeomorphic to the manifold-with-boundary $[-\pi,\pi]$.  One could instead equip these with quotient diffeologies \cite{iglesias}, in which case $\SS^n/\SO(n)$ is diffeomorphic to $\SS^m/\SO(m)$ if and only if $n=m$.  In this case, for a more interesting example, we could compare $\SO(2n)\ltimes\SS^{2n}$ with $\U(n)\ltimes\SS^{2n}$, again rotating about the axis through the north and south poles.  Since the orbits of both actions are the same, the diffeology cannot distinguish between the two orbit spaces, even though the two groupoids are not Morita equivalent (the stabilizers again are different).}) to $[-\pi,\pi]$.  However, the groupoids are not Morita equivalent: the stabilizer at any point not equal to a pole is isomorphic to $\SO(n-1)$, whereas that at one of the poles is $\SO(n)$ itself.  Since weak equivalences preserve stabilizers and orbit spaces (see, for instance, \cite[Theorem 4.3.1]{dH:orbispaces}), it is immediate that $\SO(n)\ltimes\SS^n$ is Morita equivalent to $\SO(m)\ltimes\SS^m$ if and only if $n=m$.
    \end{example}

\section{Discrete tom Dieck Fundamental Groupoid}\labell{s:fund gpd}

In order to achieve our goal and obtain a Morita-invariant twisted Bredon-Illman cohomology theory, we need to consider a twisted version of the theory.  The ordinary Bredon-Illman cohomology theory of a $G$-space $X$ is built out of fixed set data which is organized and indexed by the orbit category $\mathcal{O}_G$.  The twisted version replaces this orbit category with a fundamental groupoid category.   This category was first defined by tom Dieck \cite[Section 10.9]{tomDieck}, but can also be interpreted as a Grothendieck category of the fundamental groupoid functor over the orbit category.  In this way, it becomes a generalization of the orbit category.   This section is devoted to defining this category and developing its properties; in particular, we show that it is a Morita invariant of an action groupoid, generalizing the result for representable orbifolds of Pronk and Scull \cite{PS:fund-gpd}.   It will be used as a foundation for the definition of the Morita-invariant twisted Bredon-Illman cohomology of our main results. 

Instead of defining the full tom Dieck fundamental groupoid, and then the discrete version as a quotient of this (as done in \cite{tomDieck,PS:fund-gpd}), we follow \cite{MuMu} and define it directly.

    \begin{definition}[Discrete tom Dieck Fundamental Groupoid]\labell{d:fund-gpd}
        Fix a Lie group $G$ and a left $G$-space $X$.  The \textbf{discrete tom Dieck fundamental groupoid of $G\ltimes X$}, denoted $\Pi(G\ltimes X)$ and hereafter referred to as just the \textbf{fundamental groupoid}, is the following \emph{category}.
        \begin{enumerate}
            \setcounter{enumi}{-1}
            
            \item \textbf{Objects:} Given $H\leq G$ a closed subgroup, an object is a smooth $G$-map $x_H\colon G/H\to X$.  We will denote the point $x_H(H)$ by $x$. Note that $x$ completely determines the map $x_H$ and lies in $X^H$.
            
            \item \textbf{Arrows:} Given closed subgroups $H,K\leq G$ and objects $x_H,y_K$, let $(\widebar{g},p)$ be a pair consisting of a $G$-map $$\widebar{g}\colon G/H\to G/K\colon aH\mapsto agK$$ (the existence of which is equivalent to $g^{-1}Hg\leq K$) and a continuous path $p\colon[0,1]\to X^H$ from $x$ to $gy$.  The path $p$ determines a $G$-homotopy between $G$-maps, $G/H\times[0,1]\to X^H$ from $x_H$ to $y_K\circ\widebar{g}$.  Define an equivalence relation on these pairs as follows: $(\widebar{g}_0,p_0)\colon x_H\to y_K$ is equivalent to $(\widebar{g}_1,p_1)\colon x_H\to y_K$ if there exists a $G$-homotopy  $\Psi\colon G/H\times[0,1]\to G/K$  defined by $\Psi(aH,t)=ag_tK$ from $\widebar{g}_0$ to $\widebar{g}_1$, and $\Phi\colon[0,1]^2\to X$ an homotopy $\rel~\{0\}$ from $p_0$ to $p_1$ whose value at $(1,t)$ is equal to $y_K\circ\Psi(H,t)$.  Denote the corresponding equivalence class by $[\widebar{g}_0,p_0]$.  These classes are the arrows of $\Pi(G\ltimes X)$.   We will often depict these as in \cref{f:arrow}.

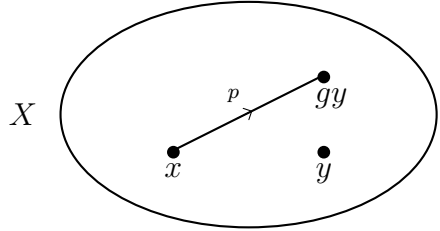
            \begin{figure}
            \centering
            \begin{tikzpicture}[thick]
                \coordinate (c);
                \draw[black] (c) ellipse (2.5 and 1.5);
                \node(X)[left of=c,xshift=-2cm]{$X$};
                \node(x)[left of=c,yshift=-0.5cm]{};
                \node(y)[right of=c,yshift=-0.5cm]{};
                \node(gy)[above of=y]{};
                \node(p)[above of=x,xshift=0.8cm,yshift=-0.25cm]{$\scriptstyle{p}$};
                \path[draw=black] ([xshift=-2mm]x.east) -- node{\firstmidarrow} ([xshift=1mm]gy.west);
                \filldraw[black] (x) circle (2pt) node[below] {$x$};
                \filldraw[black] (y) circle (2pt) node[below] {$y$};
                \filldraw[black] (gy) circle (2pt) node[below,xshift=1mm] {$gy$};
            \end{tikzpicture}
            \caption{The arrow $[\widebar{g},p]$ from $x_H$ to $y_K$.}\labell{f:arrow}
            \end{figure}
            
            \item \textbf{Units:} Given an object $x_H$, the unit is the arrow $\left[\widebar{1_G},x\right]$, where $x$ stands for the constant path at $x$.
            
            \item \textbf{Composition:} Given arrows $[\widebar{g}_0,p_0]\colon x_H\to y_K$ and $[\widebar{g}_1,p_1]\colon y_K \to z_L$, define their composition as $[\widebar{g_0g_1},p_0*(g_0p_1)]$, where we move the path $p_1$ by $g_0$ so that it can be concatenated with $p_0$ (see \cref{f:composition}).\qedhere
            \begin{figure}
            \centering
            \begin{tikzpicture}[thick]
                \coordinate (c);
                \draw[black] (c) ellipse (4.5 and 2.0);
                \node(X)[left of=c,xshift=-4cm]{$X$};
                \node(y)[yshift=-1cm]{};
                \node(g0y)[above of=y]{};
                \node(x)[left of=c,xshift=-1.5cm,yshift=-1cm]{};
                \node(z)[right of=c,xshift=1.5cm,yshift=-1cm]{};
                \node(g1z)[above of=z]{};
                \node(g0g1z)[above of=g1z]{};
                \node(p0)[above of=x,xshift=0.8cm,yshift=-0.25cm]{$\scriptstyle{p_0}$};
                \node(p1)[right of=y,xshift=0.5cm,yshift=0.25cm]{$\scriptstyle{p_1}$};
                \node(g0p1)[above of=g0y,xshift=0.8cm,yshift=-0.25cm]{$\scriptstyle{g_0p_1}$};
                \path[draw=black] (x) -- node{\firstmidarrow}(g0y);
                \path[draw=black] (y) -- node{\firstmidarrow}(g1z);
                \path[draw=black] (g0y) -- node{\firstmidarrow}(g0g1z);
                \filldraw[black] (x) circle (2pt) node[below] {$x$};
                \filldraw[black] (y) circle (2pt) node[below] {$y$};
                \filldraw[black] (z) circle (2pt) node[below] {$z$};
                \filldraw[black] (g0y) circle (2pt) node[below] {$g_0y$};
                \filldraw[black] (g1z) circle (2pt) node[below] {$g_1z$};
                \filldraw[black] (g0g1z) circle (2pt) node[below,xshift=2mm] {$g_0g_1z$};
            \end{tikzpicture}
            \caption{The composition of $[\widebar{g}_0,p_0]$ with $[\widebar{g}_1,p_1]$.}\labell{f:composition}
            \end{figure}
        \end{enumerate}
    \end{definition}

   Given an equivariant functor $\varphi\colon G\ltimes X\to H\ltimes Y$, the homomorphism $\widetilde{\varphi}\colon G\to H$ descends to a $\widetilde{\varphi}$-equivariant map $\widetilde{\varphi}_K\colon G/K\to H/\widetilde{\varphi}(K)$ for any subgroup $K\leq G$.    This can be extended to a functor of fundamental groupoids.

    \begin{proposition}[Functoriality of $\Pi$, {\cite[Proposition 4.6]{PS:fund-gpd}}]\labell{p:pi functors} 
        Given an equivariant functor $\varphi\colon G\ltimes X\to H\ltimes Y$, there is a functor $\Pi\varphi\colon\Pi(G\ltimes X)\to\Pi(H\ltimes Y)$ sending objects $x_K$ to $\varphi(x)_{\widetilde{\varphi}(K)}$ and arrows $[\widebar{g},p]$ to $\left[\widebar{\widetilde{\varphi}(g)},\varphi\circ p\right]$.
    \end{proposition}

    \begin{proof}
        To check that $\Pi(\varphi)$ is well-defined on arrows, suppose $[\widebar{g}_0,p_0]=[\widebar{g}_1,p_1]\colon (x_0)_{K_0}\to(x_1)_{K_1}$.  Then there exists $\Psi\colon G/K_0\times[0,1]\to G/K_1$  a $G$-homotopy from $\widebar{g_0}$ to $\widebar{g_1}$, and $\Phi$ an homotopy $\rel\,\{0\}$ from $p_0$ to $p_1$ whose value at $(1,t)$ is $(x_1)_{K_1}\circ\Psi(K_0,t)$.  Define an homotopy $\Psi'\colon H/\widetilde{\varphi}(K_0)\times[0,1]\to H/\widetilde{\varphi}(K_1)$ by $$\Psi'(h\widetilde{\varphi}(K_0),t):=h\widetilde{\varphi}_{K_1}(\Psi(K_0,t)).$$  Then for any $(gK_0,t)\in G/K_0\times[0,1]$, $$\varphi\circ(x_1)_{K_1}\circ\Psi(gK_0,t)=\widetilde{\varphi}(gg_t)\varphi(x_1)=\varphi(x_1)_{\widetilde{\varphi}(K_1)}\circ\Psi'\circ(\widetilde{\varphi}_{K_0}\times\id_{[0,1]})(gK_0,t).$$ Moreover, $\varphi\circ\Phi$ is an homotopy $\rel\,\{0\}$ from $\varphi\circ p_0$ to $\varphi\circ p_1$ whose value at $(1,t)$ is $\varphi(x_1)_{\widetilde{\varphi}(K_1)}\circ\Psi'(\widetilde{\varphi}_{K_0}(K_0),t)$.  The existence of $\Psi'$ and $\varphi \circ \Phi$ show  that  $\Pi(\varphi)[\widebar{g}_1, p_1] = \Pi(\varphi)[\widebar{g}_2, p_2]$ and hence  $\Pi(\varphi)$ is well-defined on arrows.

        
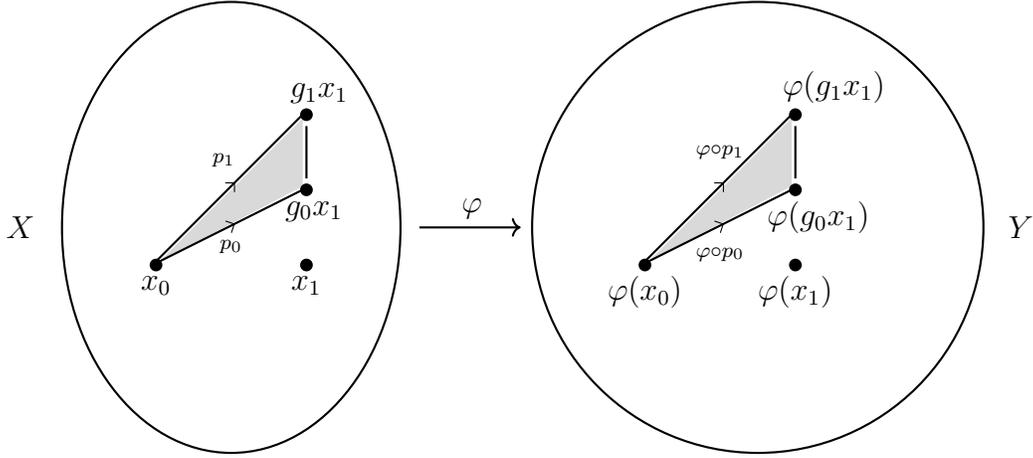
\begin{center}
        \begin{tikzpicture}[thick]
            \coordinate (c);
            \draw[black] ([xshift=-4cm]c) ellipse (2.25 and 3.0);
            \node(X)[left of=c,xshift=-5.8cm]{$X$};
            \node(x0)[left of=c,xshift=-4cm,yshift=-0.5cm]{};
            \node(x1)[right of=c,xshift=-4cm,yshift=-0.5cm]{};
            \node(g0x1)[above of=x1]{};
            \node(g1x1)[above of=g0x1]{};
            \node(p0)[above of=x0,xshift=1cm,yshift=-0.75cm]{$\scriptstyle{p_0}$};
            \node(p1)[above of=x0,xshift=0.9cm,yshift=0.4cm]{$\scriptstyle{p_1}$};
            \fill[gray!30] ([xshift=0.5mm,yshift=0.5mm]x0.center) to ([xshift=-0.5mm,yshift=0.5mm]g0x1.center) -- ([xshift=-0.5mm,yshift=-0.5mm]g1x1.center);
            \path[draw=black] ([xshift=-2mm]x0.east) -- node{\firstmidarrow} ([xshift=1mm]g0x1.west);
            \path[draw=black] ([xshift=-2mm]x0.east) -- node{\secondmidarrow} ([xshift=1mm]g1x1.west);
            \path[draw=black] (g0x1) -- (g1x1);
            \filldraw[black] (x0) circle (2pt) node[below] {$x_0$};
            \filldraw[black] (x1) circle (2pt) node[below] {$x_1$};
            \filldraw[black] (g0x1) circle (2pt) node[below,xshift=1mm] {$g_0x_1$};
            \filldraw[black] (g1x1) circle (2pt) node[above,xshift=1mm] {$~g_1x_1$};

            \node(phisrc)[left of=c,xshift=-0.65cm]{};
            \node(phitrg)[right of=c,xshift=-1cm]{};
            \draw[->] (phisrc) -- (phitrg);
            \node(phi)[left of=c,xshift=0.2cm,yshift=0.25cm]{$\varphi$};
            
            \draw[black] ([xshift=3cm]c) ellipse (3.0 and 3.0);
            \node(Y)[right of=c,xshift=5.5cm]{$Y$};
            \node(phix0)[left of=c,xshift=2.5cm,yshift=-0.5cm]{};
            \node(phix1)[right of=c,xshift=2.5cm,yshift=-0.5cm]{};
            \node(phig0x1)[above of=phix1]{};
            \node(phig1x1)[above of=phig0x1]{};
            \node(phip0)[above of=phix0,xshift=1cm,yshift=-0.85cm]{$\scriptstyle{\varphi\circ p_0}$};
            \node(phip1)[above of=phix0,xshift=1cm,yshift=0.5cm]{$\scriptstyle{\varphi\circ p_1}$};
            \fill[gray!30] ([xshift=0.5mm,yshift=0.5mm]phix0.center) to ([xshift=-0.5mm,yshift=0.5mm]phig0x1.center) -- ([xshift=-0.5mm,yshift=-0.5mm]phig1x1.center);
            \path[draw=black] ([xshift=-2mm]phix0.east) -- node{\firstmidarrow} ([xshift=1mm]phig0x1.west);
            \path[draw=black] ([xshift=-2mm]phix0.east) -- node{\secondmidarrow} ([xshift=1mm]phig1x1.west);
            \path[draw=black] (phig0x1) -- (phig1x1);
            \filldraw[black] (phix0) circle (2pt) node[below] {$\varphi(x_0)$};
            \filldraw[black] (phix1) circle (2pt) node[below] {$\varphi(x_1)$};
            \filldraw[black] (phig0x1) circle (2pt) node[below,xshift=1mm] {$\quad\varphi(g_0x_1)$};
            \filldraw[black] (phig1x1) circle (2pt) node[above,xshift=1mm] {$\quad\quad\varphi(g_1x_1)$};            
        \end{tikzpicture}
        \end{center}

        Units are sent to units by $\Pi\varphi$, and unraveling definitions yields that $\Pi\varphi$ respects composition of arrows; see \cite[Proposition 4.6]{PS:fund-gpd} for details. 
    \end{proof}

    \begin{corollary}\labell{c:pi functors}
        $\Pi$ is a covariant functor from $\ActGpd$ to $\Cat$, the category of small categories. 
    \end{corollary}

    \begin{proof}
        The identity functor on an action groupoid $G \ltimes X$ is an equivariant functor, with associated group homomorphism $\id_G$.  It follows that $\Pi$ sends the identity functor to the identity functor of $\Pi(G\ltimes X)$.

        Given equivariant functors $\varphi\colon G\ltimes X\to H\ltimes Y$ and $\psi\colon H\ltimes Y\to K\ltimes Z$, for any $x\in X$ and subgroup $L\leq G$,
        $$\psi\circ\varphi(x)_{\widetilde{\psi}\circ\widetilde{\varphi}(L)}=\psi(\varphi(x))_{\widetilde{\psi}(\widetilde{\varphi}(L))}$$ and for any arrow $[\widebar{g},p]$ of $\Pi(G\ltimes X)$,
        $$\left(\overline{\widetilde{\psi}\circ\widetilde{\varphi}(g)},(\psi\circ\varphi)\circ p\right)=\left(\overline{\widetilde{\psi}(\widetilde{\varphi}(g))},\psi\circ(\varphi\circ p)\right).$$
        The result follows.        
    \end{proof}

   $\Pi$ acts on natural transformations as well, and in fact defines a $2$-functor from $\ActGpd$ to $\Cat$, which we prove below.  In \cite[Proposition 4.7, Theorem 4.8]{PS:fund-gpd}, the authors show that the non-discrete tom Dieck fundamental groupoid (for which they use the notation $\Pi$) induces a pseudofunctor from $\ActGpd$ to bicategories.  Our discrete version is obtained via a quotient on the arrows of these fundamental groupoids, similar to how the homotopy category is a quotient on the continuous maps in the category of topological spaces. Passing to the discrete version is functorial, and in fact sends the pseudo-natural transformations obtained in \cite[Proposition 4.7]{PS:fund-gpd} to \emph{natural} transformations in $\Cat$:

    \begin{theorem}[2-Functoriality of $\Pi$]\labell{t:pi nat transf}
        Given equivariant functors $\varphi,\psi\colon G\ltimes X\to H\ltimes Y$, and a smooth natural transformation $S\colon \varphi\Rightarrow\psi$, there is a natural transformation $\Pi S\colon\Pi\varphi\Rightarrow\Pi\psi$ sending an object $x_K\colon G/K\to X$ to the arrow
        $$\xymatrix{
            \varphi(x)_{\widetilde{\varphi}(K)} \ar[rrr]^{\left[\overline{\widetilde{S}(x)^{-1}},\,\varphi(x)\right]} &&& \psi(x)_{\widetilde{\psi}(K)} \\
        }$$ where $\widetilde{S}\colon X\to H$ is a smooth map such that $S(x)=\left(\widetilde{S}(x),\varphi(x)\right)$. In fact, $\Pi$ is a $2$-functor from $\ActGpd$ to $\Cat$.
    \end{theorem}

    This follows immediately from the proof of \cite[Proposition 4.7]{PS:fund-gpd} by  passing  to the discrete fundamental groupoid, where  the $2$-commutative diagrams in the proof of pseudo-naturality become commutative diagrams, proving naturality.

\mute{
    \begin{proof}
        Given an arrow $\left[\overline{g},p\right]$ between $x_K$ and $y_L$ in $\Pi(G\ltimes X)$, we need to show that the following diagram is commutative:
        $$\xymatrix{
            \varphi(x)_{\widetilde{\varphi}(K)} \ar[rr]^{\left[\overline{\widetilde{\varphi}(g)},\,\varphi\circ p\right]} \ar[d]_{\left[\overline{\widetilde{S}(x)^{-1}},\,\varphi(x)\right]} & & \varphi(y)_{\widetilde{\varphi}(L)} \ar[d]^{\left[\overline{\widetilde{S}(y)^{-1}},\,\varphi(y)\right]} \\
            \psi(x)_{\widetilde{\psi}(K)} \ar[rr]_{\left[\overline{\widetilde{\psi}(g)},\,\psi\circ p\right]} & & \psi(y)_{\widetilde{\psi}(L)}. \\
        }$$
        That is, we need to show that the two compositions 
        $$\left[\overline{\widetilde{\varphi}(g)\widetilde{S}(y)^{-1}},\,\varphi\circ p\right]\quad \text{and} \quad \left[\overline{\widetilde{S}(x)^{-1}\widetilde{\psi}(g)},\,\widetilde{S}(x)^{-1}\psi\circ p\right]$$ are equal.  This is accomplished via the homotopies $$\Psi\colon H/\widetilde{\varphi}(K)\times[0,1]\to H/\widetilde{\psi}(L)\colon(\widetilde{\varphi}(aK),t)\mapsto a\widetilde{S}(p(t))^{-1}\widetilde{\psi}(g)\widetilde{\psi}(L)$$ and $$\Phi\colon[0,1]^2\to X\colon (s,t)\mapsto \widetilde{S}(p(st))^{-1}\psi\circ p(s)$$ as in \cref{d:fund-gpd}.
    \end{proof}
}
    The following is an extension of \cite[Theorem 5.1]{PS:fund-gpd} from orbifold groupoids to action groupoids of compact Lie group actions.  In fact, using bibundles allows us to prove the extension almost directly from the results in \cite{PS:fund-gpd}.

    \begin{theorem}[Morita Invariance of the Fundamental Groupoid]\labell{t:fund gpd morita}
        Given Morita equivalent action groupoids $G\ltimes X$ and $H\ltimes Y$, the fundamental groupoids $\Pi(G\ltimes X)$ and $\Pi(H\ltimes Y)$ are equivalent as categories.
    \end{theorem}

    \begin{proof}
        Since $G\ltimes X$ and $H\ltimes Y$ are Morita equivalent, there is a biprincipal bibundle $Z$ between these; moreover, this bibundle has the format described in \cref{p:bibundle form,c:bibundle form}.  Thus this bibundle is a $(G\times H)$-space, in which the left and right anchor maps $\lambda\colon Z\to X$ and $\rho\colon Z\to Y$ induce group homomorphisms $\widetilde{\lambda}$ and $\widetilde{\rho}$ that are the projection maps onto $G$ and $H$, resp. 

        Since both $\lambda$ and $\rho$ are principal bundles, both of these maps take the form of \cite[Proposition 5.6]{PS:fund-gpd}, which states that they have equivalent fundamental groupoids. 
    \end{proof}

    We now continue the investigation of the examples considered in the previous section.

    \begin{example}[{\bf Groupoid A}]\labell{x:a2}
        Recall $\ZZ/2\ltimes\SS^1$ of  \cref{x:a1}.   Its fundamental groupoid $\Pi(\ZZ/2\ltimes\SS^1)$ has objects 
        \begin{itemize}
            \item $x_{\langle 1\rangle}\colon(\ZZ/2)/\langle 1\rangle\to\SS^1$ for each $x\in\SS^1$, and
            \item $N_{\ZZ/2}\colon(\ZZ/2)/(\ZZ/2)\to\SS^1$ and $S_{\ZZ/2}\colon(\ZZ/2)/(\ZZ/2)\to\SS^1$ corresponding to the (fixed) north $N$ and south $S$ poles.
        \end{itemize}
        The arrows are given by compositions of the following: 
        \begin{itemize}
            \item Each element $g$ of $\ZZ/2$ induces an arrow  $(\widebar{g}, x)$ where the path is constant, and the arrow goes from $x_{\langle 1\rangle}$ to $(gx)_{\langle 1\rangle}$ moving underlying points along their orbits.
            \item For each $x_0,x_1\in\SS^1$ and element of $\pi_1(\SS^1;x_0)$, there is an arrow $[\widebar{1},p]\colon(x_0)_{\langle 1\rangle}\to(x_1)_{\langle 1\rangle}$ where $p$ is a path from $x_0$ to $x_1$ that is homotopic $\rel\,\{0\}$ to a representative of the element of $\pi_1(\SS^1;x_0)$.
            \item For each $x\in\SS^1$ and element of $\pi_1(\SS^1;x)$, there is an arrow $[\widebar{1},p]\colon x_{\langle 1\rangle}\to N_{\ZZ/2}$ where $p$ is a path from $x$ to $N$ that is homotopic $\rel\,\{0\}$ to a representative of the  element of $\pi_1(\SS^1;x)$.
            \item Similarly, for each $x\in\SS^1$ and element of $\pi_1(\SS^1;x)$ there is an arrow $[\widebar{1},p]\colon x_{\langle 1\rangle}\to S_{\ZZ/2}$.\qedhere
        \end{itemize}
    \end{example} 

    \begin{example}[{\bf Groupoid B}]\labell{x:b2}  
        The Klein bottle $Y=\U(1)\times_{\ZZ/2}\SS^1$ with the action of $\U(1)$ as in \cref{x:b1} has a fundamental groupoid $\Pi(\U(1)\ltimes(\U(1)\times_{\ZZ/2}\SS^1))$ consisting of objects 
        \begin{itemize}
            \item $[e^{i\alpha},e^{i\beta}]_{\langle 1\rangle}$ for each $[e^{i\alpha},e^{i\beta}]$ in $Y$,
            \item $[e^{i\alpha},N]_{\langle-1\rangle}$ for each $\alpha\in[0,2\pi)$, and
            \item $[e^{i\alpha},S]_{\langle-1\rangle}$ for each $\alpha\in[0,2\pi)$.
        \end{itemize}
        The arrows are compositions of arrows of the following: 
        \begin{itemize}
            \item     Each element $g$ of  $\U(1)$ induces an arrow  $(\widebar{g}, x)$ where the path is constant, and the arrow goes from $x_{\langle 1\rangle}$ to $(gx)_{\langle 1\rangle}$ moving underlying points along their orbits.     
            \item For each $y_0:=[e^{i\alpha_0},e^{i\beta_0}]$ and $y_1:=[e^{i\alpha_1},e^{i\beta_1}]\in Y$ and element of $\pi_1(Y;y_0)$, there is an arrow $[\widebar{1},p]$ where $p$ is a path from $y_0$ to $y_1$ homotopic $\rel\,\{0\}$ to a representative of the element of $\pi(Y;y_0)$.
            \item For each $y_0:=[e^{i\alpha_0},e^{i\beta_0}]$ and $y_1:=[e^{i\alpha_1},N]\in Y$ and element of $\pi_1(Y;y_0)$, there is an arrow $[\widebar{1},p]\colon(y_0)_{\langle1\rangle}\to(y_1)_{\langle-1\rangle}$ where $p$ is a path from $y_0$ to $y_1$ homotopic $\rel\,\{0\}$ to a representative of the element of $\pi_1(Y;y_0)$.
            \item Similarly, for each $y_0:=[e^{i\alpha_0},e^{i\beta_0}]$ and $y_1:=[e^{i\alpha_1},S]\in Y$ and element of $\pi_1(Y;y_0)$, there is an arrow $[\widebar{1},p]\colon(y_0)_{\langle1\rangle}\to(y_1)_{\langle-1\rangle}$.\qedhere
        \end{itemize}
    \end{example}    

    \begin{example}[{\bf Groupoid C}]\labell{x:c2}
        The groupoid $D_4\ltimes\SS^1$ as in \cref{x:c1} has fundamental groupoid $\Pi(D_4\ltimes\SS^1)$ with objects
        \begin{itemize}
            \item $x_{\langle 1\rangle}\colon D_4/\langle1\rangle\to\SS^1$ for each $x\in\SS^1$, 
            \item the four ``poles'' $N_{\ZZ/2\times 1}\colon D_4/(\ZZ/2\times 1)\to\SS^1$, $S_{\ZZ/2\times 1}\colon D_4/(\ZZ/2\times 1)\to\SS^1$, $W_{1\times\ZZ/2}\colon D_4/(1\times\ZZ/2)\to\SS^1$, and $E_{1\times\ZZ/2}\colon D_4/(1\times\ZZ/2)\to\SS^1$.
        \end{itemize}
        The arrows are given by compositions of the following:
        \begin{itemize}
            \item     Each element $g$ of  $D_4$ induces an arrow  $(\widebar{g}, x)$ where the path is constant, and the arrow goes from $x_{\langle 1\rangle}$ to $(gx)_{\langle 1\rangle}$ moving underlying points along their orbits.    
            \item For any $x_0,x_1\in\SS^1$ and element of $\pi_1(\SS^1;x_0)$, there is an arrow $[\widebar{1},p]$ where $p$ is a path from $x_0$ to $x_1$ that is homotopic $\rel\,\{0\}$ to a representative of the element of $\pi_1(\SS^1;x_0)$.
            \item For each $x\in\SS^1$ and element of $\pi_1(\SS^1;x)$, there are arrows $[\widebar{1},p]\colon x_{\langle1\rangle}\to N_{\ZZ/2\times 1}$, $[\widebar{1},p]\colon x_{\langle1\rangle}\to S_{\ZZ/2\times 1}$, $[\widebar{1},p]\colon x_{\langle1\rangle}\to W_{1\times\ZZ/2}$, and
            $[\widebar{1},p]\colon x_{\langle1\rangle}\to E_{1\times\ZZ/2}$, where $p$ is a path from $x$ to the corresponding endpoint that is homotopic $\rel\,\{0\}$ to the element of $\pi_1(\SS^1;x)$.\qedhere       
        \end{itemize}
    \end{example}
    
    \begin{example}\labell{x:ab pi}
        Recall also from \cref{x:ab} and \cref{x:ac} that Groupoids A, B and C are all Morita equivalent.   Thus their fundamental groupoids are all equivalent as categories by \cref{t:fund gpd morita}.  For instance, we can verify this directly for the functor $\Pi\varphi\colon\Pi(\ZZ/2\ltimes\SS^1)\to\Pi(\U(1)\ltimes Y)$, where $\varphi$ is the equivariant functor $\varphi\colon\ZZ/2\ltimes\SS^1\to\U(1)\ltimes Y$ given in \cref{x:ab}.  Indeed, essential surjectivity and faithfulness are straightforward to check, and fullness follows from the fact that given an arrow $\left[\overline{e^{i\theta}},[e^{i\alpha(t)},e^{i\beta(t)}]\right]$ between objects in the image of $\Pi\varphi$, this arrow is equal to an arrow in the image of $\Pi\varphi$.  Specifically, it is equal to the arrow  $\left[\overline{e^{i(\theta-\alpha(1))}},[1,e^{i\beta(t)}]\right]$,  since $[e^{i\alpha(t)},e^{i\beta(t)}]=e^{i\alpha(t)}\cdot[1,e^{i\beta(t)}]$ and $\alpha(0),\alpha(1),\theta\in\pi\ZZ$.
    \end{example}
       
    \begin{example}[{\bf Groupoid D}]\labell{x:d3}  Recall also that the groupoids of \cref{x:d1} are \emph{not} Morita equivalent for $n \neq m$.   The fundamental groupoids are also distinct, since the objects  $x_{\langle 1\rangle}\colon \SO(n)/\langle1\rangle\to\SS^n$ will have arrows from  $x_{\langle 1\rangle}$ to itself for any element of the stabilizer $\SO(n-1)$.   Thus $\Pi(\SO(n)\ltimes\SS^n)$ encodes the stabilizers, and so $\Pi(\SO(n)\ltimes\SS^n)$ is equivalent to $\Pi(\SO(m)\ltimes\SS^m)$ if and only if $n=m$.
    \end{example}

\section{Coefficient Systems and Bibundles}\labell{s:coeff systems}

In this section, we will be studying coefficient systems, the coefficients for twisted Bredon-Illman cohomology.   Coefficients for ordinary Bredon-Illman cohomology are defined as functors from the orbit category $\mathcal{O}_G$ to abelian groups.  Ordinary coefficients do not give a Morita-invariant theory without additional restrictions, as illustrated by the example in \cite[Example 5.3]{PS:translation}.  The twisted version of the cohomology addresses this issue by taking our indexing category to be the fundamental groupoid of \cref{s:fund gpd}.   In this theory, twisted coefficients are defined as  functors from the fundamental groupoid to abelian groups. 

When we get to our main result, the Morita invariance of our cohomology theory, we will want to pass from a coefficient system on one action groupoid to a coefficient system on a Morita equivalent groupoid.   As we are using bibundles to encode our Morita equivalences, this means that we want to be able to transfer  coefficient systems from one action groupoid to another across a biprincipal bibundle morphism.   This section develops the theory we need to do this.  

Our first goal is to prove the following.  
   
    \begin{theorem}[Right Inverses to $\Pi\lambda$]\labell{p:right inverse}
        Let $G$ and $H$ be compact Lie groups and let $G\ltimes X\ot{\lambda}Z\uto{\rho} H\ltimes Y$ be a bibundle.
        \begin{enumerate}
            \item\labell{i:right inverse}There is a right inverse $\Sigma\colon\Pi(G\ltimes X)\to\Pi((G\times H)\ltimes Z)$ to $\Pi\lambda$.
            \item\labell{i:right inverse choices}Any two such right inverses of $\Pi\lambda$ differ by a natural isomorphism.
        \end{enumerate}
    \end{theorem}

    Before proving \cref{p:right inverse}, we note that as a consequence of this theorem, given a bibundle $$P:=(G\ltimes X\ot{\lambda}Z\uto{\rho} H\ltimes Y),$$  $\Pi P\colon\Pi(G\ltimes X)\to\Pi(H\ltimes Y)$ will only be defined up to natural isomorphism.   This will be sufficient to give us the results we need on coefficient systems later in this section.  

    To prove \cref{p:right inverse}, we will begin by proving several lemmas relating fixed sets in  $X$  to those in a bibundle $Z$.

    \begin{lemma}\labell{l:gamma}
        Let $\lambda\colon Z\to X$ be a $G$-equivariant principal $H$-bundle in which the $G$- and $H$-actions on $Z$ commute.
        \begin{enumerate}
            \item\labell{i:gamma}Given a closed subgroup $K\leq G$, a point $x\in X^K$, and a point $z\in\lambda^{-1}(x)$, there is a unique subgroup $\Gamma^K_z$ of $G\times H$ such that $\pr_1(\Gamma^K_z)=K$ and $z\in Z^{\Gamma^K_z}$.
            \item\labell{i:gamma conj}Given another $z'\in\lambda^{-1}(x)$, the group  $\Gamma^K_{z'}$ is conjugate to $\Gamma^K_z$.  
            \item\labell{i:gamma stab}If $K=\Stab_G(x)$, then $\Gamma^K_z=\Stab_{G\times H}(z)$. 
        \end{enumerate}            
    \end{lemma}

    \begin{proof}
        Fix $K$, $x$, and $z$ as in \cref{i:gamma}.  For any $k\in K$, since $kz\in\lambda^{-1}(x)$ and $\lambda\colon Z\to X$ is $H$-principal, it follows that there is a unique $\zeta^K_z(k)\in H$ such that $kz\zeta_z(k)^{-1}=z$.  In fact, for any $k_0,k_1\in K$, $$(k_1k_0)z\zeta^K_z(k_1k_0)^{-1}=z=k_1\left(k_0z\zeta^K_z(k_0)^{-1}\right)\zeta^K_z(k_1)^{-1};$$ from principality of $\lambda$ it follows that $\zeta^K_z(k_1k_0)=\zeta^K_z(k_1)\zeta^K_z(k_0)$.   Since the $H$-principal bundle $\lambda\colon Z\to X$ admits a diffeomorphism $A\colon H\times Z\to Z\ftimes{\lambda}{\lambda}Z$ sending $(h,z)$ to $(z,zh^{-1})$, the division map $d=\pr_1\circ A^{-1}\colon Z\ftimes{\lambda}{\lambda}Z\to H\colon(z,zh^{-1})\mapsto h$ is smooth, from which it follows that  $\zeta^K_z\colon K\to H$ is smooth.  Thus $\zeta^K_z$ is a Lie group homomorphism.  Define $\Gamma^K_z$ to be the graph of $\zeta^K_z$, which is a closed subgroup of $G\times H$.  Since $\pr_1(\Gamma^K_z)=K$, uniqueness follows from principality.  This proves \cref{i:gamma}.

        \cref{i:gamma conj} follows again from principality:   if $z, z'$ are two lifts of $x$, then there is a unique $b\in H$ such that $z'=zb^{-1}$.  Then by definition, $z = k z \zeta^K_z(k)^{-1}$ and so $zb^{-1}  = kz\zeta^K_z(k)^{-1} b^{-1} = kzb^{-1} b \zeta^K_z(k)^{-1} b^{-1}  $.   Therefore $\zeta^K_{z'}(k) = b \zeta^K_z(k)b^{-1}$. 

        Suppose $K=\Stab_G(x)$ and let $(g,h)\in\Stab_{G\times H}(z)$.  Then $g\in K$ and by principality, $h=\zeta^K_z(g)$.  Thus $(g,h)\in\Gamma^K_z$.  It follows that $\Gamma^K_z=\Stab_{G\times H}(z)$.    This proves \cref{i:gamma stab}.
    \end{proof}
    
    \begin{remark}\labell{r:gamma}  We make note of the following specifics in our proof, which will be useful going forward:  
        \begin{itemize}
            \item  The subgroup $\Gamma^K_z$ is defined as the  graph of a group homomorphism $\zeta^K_z\colon K\to H$ defined by $$kz\zeta^K_z(k)^{-1}=z.$$ 
            \item If $z, z'$ are two lifts of $x$,  then   $$\Gamma^K_{z'}=(1_G,b)\,\Gamma^K_z\,(1_G,b^{-1})$$ where $b \in  H$ is the unique element  such that $z'=zb^{-1}$.  
            \item By principality,  for any closed subgroup $L\leq\Stab_G(x)$, the equality $\zeta^L_z=\zeta^{\Gamma_z}_z|_L$ holds.   Hence  we will typically drop the superscript $K$ or $L$ from $\zeta_z$ and $\Gamma_z$ unless it is needed for clarity, henceforth. \qedhere
        \end{itemize}
    \end{remark}

    \begin{lemma}\labell{l:lifting}
        Let $G$ and $H$ be compact Lie groups and $\lambda\colon Z\to X$ a $G$-equivariant principal $H$-bundle in which the $G$- and $H$-actions on $Z$ commute.  Let $K$ be a closed subgroup of $G$ and fix $x_0\in X^K$, $z_0\in\lambda^{-1}(x_0)$, and  define  $\Gamma_{z_0} \leq G \times H$ the pre-image of $K$ fixing $z_0$ as in  \cref{l:gamma}. If $C$ is the connected component of $Z^{\Gamma_{z_0}}$ containing $z_0$, and $D$ is the connected component of $X^K$ containing $x_0$, then $\lambda|_C\colon C\to D$ is a locally trivial fibration.  Consequently, $\lambda|_C$ has the homotopy lifting property.
    \end{lemma}

    \begin{proof}
        By Ehresmann's Lemma \cite{ehresmann}, $\lambda|_C$ is a locally trivial fibration if it is a proper surjective submersion.  We begin with surjective submersivity.

        By \cref{t:bierstone}, for any $x\in X^K$ there is a $\Stab_G(x)$-invariant (hence $K$-invariant) open neighborhood $U$ of $x$ and a $\Stab_G(x)$-equivariant (hence $K$-equivariant) principal $H$-bundle diffeomorphism $\psi\colon\lambda^{-1}(U)\to U\times H$.    The action of $\Stab_G(x)\times H$ on $U\times H$ is given by 
        $$(\ell,h)\cdot(x,h')=\left(\ell\cdot x,hh'\zeta_{z}(\ell)^{-1}\right)$$
where $z=\psi^{-1}(x,1_H)$. 
        
        Fix $x\in D$, and let $p\colon[0,1]\to X$ be a smooth path in $D$ with $p(0)=x_0$ and $p(1)=x$.  Choose a partition $0=t_0<\dots<t_n=1$ of $[0,1]$ such that for each $i$, setting $x_i:=p(t_i)$, there is a $\Stab(x_i)$-invariant open neighborhood $U_i$ of $x_i$ and a $\Stab(x_i)$-equivariant $H$-bundle diffeomorphism $\psi_i\colon\lambda^{-1}(U_i)\to U_i\times\lambda^{-1}(x_i)$; choose these $t_i$ and $U_i$ so that $\{U_i\}$ covers the image of $p$.  For each $i=0,\dots,n-1$, fix a point $t_{i+1/2}\in(t_i,t_{i+1})$ such that $p(t_{i+1/2})\in U_i\cap U_{i+1}$.
        
        Define the smooth path $\widetilde{p}_0\colon[0,t_{1/2}]\to \lambda^{-1}(U_0)$ by $\widetilde{p}_0(t):=\psi_0^{-1}(p(t),1_H)$.  For any $(k,\zeta_{z_0}(k))\in\Gamma_{z_0}$, $$(k,\zeta_{z_0}(k))\cdot\widetilde{p}(t)=\psi_0^{-1}(kp(t),\zeta_{z_0}(k)1_H\zeta_{z_0}(k)^{-1})=\psi_0^{-1}(p(t),1_H)=\widetilde{p}(t).$$  So $\widetilde{p}$ is a smooth local lift of $p|_{[0,t_{1/2}]}$ to $C$.
        
        Define $\widetilde{p}_1\colon[t_{1/2},t_{3/2}]\to \lambda^{-1}(U_1)$ by $\widetilde{p}_1(t):=\psi_1^{-1}(p(t),1_H)$.  There is a unique $h_1\in H$ such that $$\widetilde{p}_0(t_{1/2})h_1^{-1}=\widetilde{p}_1(t_{1/2}).$$  It now follows from \cref{r:gamma} that $$(k,\zeta_{z_0}(k))\cdot\left(\widetilde{p}_1(t)h_1\right)=\psi_1^{-1}(p(t),\zeta_{z_0}(k)h_1^{-1}\zeta_{z_1}^{-1}(k))=\widetilde{p}_1(t)h_1;$$ this is a lift of $p|_{[t_{1/2},t_{3/2}]}$ to $C$.  Continuing a finite number of times, we construct a lift of all of $p$ to a smooth path in $C$.  This shows that $\lambda|_C$ is surjective and submersive.

        Fix $x\in D$ and let $U$ be a $\Stab(x_\infty)$-invariant open neighborhood and $\psi\colon \lambda^{-1}(U)\to U\times H$ a $\Stab(x_\infty)$-equivariant $H$-bundle diffeomorphism. Let $(z_i)$ be a sequence in $C\cap\lambda^{-1}(U)$ such that $x_i:=\lambda(z_i)$ converges to $x_\infty$.   There exists a sequence $(h_i)$ in $H$ such that for each $i$, $$z_i=\psi^{-1}(x_i,h_i).$$  Since $H$ is compact, there is a subsequence $\left(h_{i_j}\right)$ of $(h_i)$ that converges to some $h_\infty\in H$.  Then $z_\infty:=\psi^{-1}(x_\infty,h_\infty)$ is the limit of $\left(z_{i_j}\right)$.  Since $C$ is closed in $Z^{\Gamma_z}$, which in turn is closed in $Z$, we have that $C$ is closed in $Z$ and $z_\infty\in C$.  Thus $\lambda|_C$ is a proper map.  This completes the proof.
    \end{proof}    

    We are now ready to prove \cref{p:right inverse}.

    \begin{proof}[Proof of \cref{p:right inverse}]
    We begin by defining the right inverse $\Sigma$.  
        For each $x\in X$, choose $z\in\lambda^{-1}(x)$.  Given  $x_K$ of $\Pi(G\ltimes X)$, we define $\Sigma(x_K):=z_{\Gamma_z}$ in $\Pi((G\times H)\ltimes Z)$, where $\Gamma_z$  is the lift of $K$ fixing $z$ as in \cref{i:gamma} of \cref{l:gamma}.  

        For an arrow $\left[\widebar{g},p\right]\colon(x_0)_{K_0}\to(x_1)_{K_1}$, and choices $z_i\in\lambda^{-1}(x_i)$ ($i=0,1$), define $\Sigma\left[\widebar{g},p\right]$ to be $\left[\overline{(g,h)},\widetilde{p}\right]$, where $\widetilde{p}$ is a lift of $p$ to $Z^{\Gamma_z}$ starting at $z_0$ (which exists by \cref{l:lifting}), and $h$ is the unique element of $H$ such that $\widetilde{p}(1)=gz_1h^{-1}$.  

        We need to show that this is well-defined in spite of the choice of lift $\widetilde{p}$.   So suppose that $\widetilde{p}'$ is a different lift of $p$ contained in $Z^{\Gamma_z}$ starting at $z_0$, and $h'$ is the unique element of $H$ such that $\widetilde{p}'(1)=gz_1(h')^{-1}$.  We can define an homotopy  $\rel\;\{0\}$,  $\widetilde{\Phi}$,  from $\widetilde{p}$ to $\widetilde{p}'$ contained in $Z^{\Gamma_z}$, by lifting the trivial homotopy from $p$ to itself via \cref{l:lifting}.  Since $\lambda$ is a principal $H$-bundle, there is a path $h_s$ in $H$ starting at $1_H$ such that $\widetilde{\Phi}(s,1)=\widetilde{p}(1)h_s^{-1}$.  It follows that $$\left[\overline{\left(g,h\right)},\widetilde{p}\right]=\left[\overline{\left(g,h'\right)},\widetilde{p}'\right],$$ and so $\Sigma([\widebar{g},p])$ is well-defined.

        It is immediate that $\Sigma$ sends units to units.  We check that $\Sigma$ respects composition: for $i=0,1,2$ let $(x_i)_{K_i}$ be objects of $\Pi(G\ltimes X)$, and for $i=0,1$ let $[\widebar{g_i},p_i]\colon (x_i)_{K_i}\to (x_{i+1})_{K_{i+1}}$ be arrows,  with composition defined by  $$[\widebar{g_0},p_0][\widebar{g_1},p_1]=[\overline{g_0g_1},p_0*g_0p_1];$$  illustrated in \cref{f:sigma comp}.  For choices $z_i\in\lambda^{-1}(x_i)$ ($i=0,1,2$), and for $i=0,1$ let $\widetilde{p}_i$ be a lift of $p_i$ starting at $z_i$ and ending at $\widetilde{z}_i$.  For $i=0,1$ there exist unique $\widetilde{h_i}\in H$ such that $g_iz_{i+1}(\widetilde{h}_i)^{-1}=\widetilde{z}_i$.  The path $\widetilde{p}_0*g_0\widetilde{p}_1\widetilde{h}_0^{-1}$ is a lift of $p_0*g_0p_1$ starting at $z_0$ and ending at $g_0\widetilde{z}_1\widetilde{h}_0^{-1}$. Since $\lambda$ restricts to a map with the homotopy lifting property on  fixed sets by \cref{l:lifting}, any other path $\widetilde{q}$ lifting $p_0*g_0p_1$ starting at $z_0$ and ending at $g_0g_1z_2\ell^{-1}$ for some unique $\ell$ will satisfy $$\left[\overline{(g_0g_1,\ell)},\widetilde{q}\right]=\left[\overline{(g_0g_1,h_0h_1)},\widetilde{p}_0*g_0\widetilde{p}_1\widetilde{h}_0^{-1}\right],$$ by the well-definedness argument above.  It follows that        $$\Sigma\left(\left[\widebar{g_0},p_0\right]\left[\widebar{g_1},p_1\right]\right)=\Sigma\left(\left[\widebar{g_0},p_0\right]\right)\Sigma\left(\left[\widebar{g_1},p_1\right]\right).$$
        This proves that $\Sigma$ respects composition, and hence is a functor, proving \cref{i:right inverse}.

        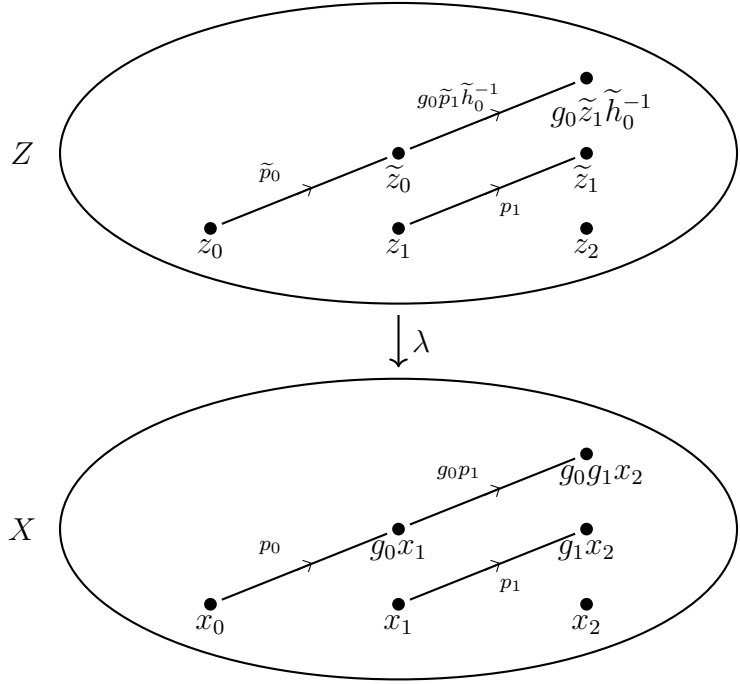
\begin{figure}
        \centering
            \begin{tikzpicture}[thick]
                \coordinate (c);
                \draw[black] (c) ellipse (4.5 and 2.0);
                \node(X)[left of=c,xshift=-4cm]{$X$};
                \node(x1)[yshift=-1cm]{};
                \node(g0x1)[above of=x1]{};
                \node(x0)[left of=c,xshift=-1.5cm,yshift=-1cm]{};
                \node(x2)[right of=c,xshift=1.5cm,yshift=-1cm]{};
                \node(g1x2)[above of=x2]{};
                \node(g0g1x2)[above of=g1x2]{};
                \node(p0)[above of=x0,xshift=0.8cm,yshift=-0.25cm]{$\scriptstyle{p_0}$};
                \node(p1)[right of=x1,xshift=0.5cm,yshift=0.25cm]{$\scriptstyle{p_1}$};
                \node(g0p1)[above of=g0x1,xshift=0.8cm,yshift=-0.25cm]{$\scriptstyle{g_0p_1}$};
                \path[draw=black] (x0) -- node{\firstmidarrow}(g0x1);
                \path[draw=black] (x1) -- node{\firstmidarrow}(g1x2);
                \path[draw=black] (g0x1) -- node{\firstmidarrow}(g0g1x2);
                \filldraw[black] (x0) circle (2pt) node[below] {$x_0$};
                \filldraw[black] (x1) circle (2pt) node[below] {$x_1$};
                \filldraw[black] (x2) circle (2pt) node[below] {$x_2$};
                \filldraw[black] (g0x1) circle (2pt) node[below] {$g_0x_1$};
                \filldraw[black] (g1x2) circle (2pt) node[below] {$g_1x_2$};
                \filldraw[black] (g0g1x2) circle (2pt) node[below,xshift=2mm] {$g_0g_1x_2$};

                \draw[black] (c)[yshift=5cm] ellipse (4.5 and 2.0);
                \node(Z)[left of=c,xshift=-4cm,yshift=5cm]{$Z$};
                \node(z1)[yshift=4cm]{};
                \node(z0-tilde)[above of=x1,yshift=5cm]{};
                \node(z0)[left of=c,xshift=-1.5cm,yshift=4cm]{};
                \node(z2)[right of=c,xshift=1.5cm,yshift=4cm]{};
                \node(z1-tilde)[above of=z2]{};
                \node(g0h0z1-tilde)[above of=z1-tilde]{};
                \node(p0-tilde)[above of=z0,xshift=0.8cm,yshift=-0.25cm]{$\scriptstyle{\widetilde{p}_0}$};
                \node(p1-tilde)[right of=z1,xshift=0.5cm,yshift=0.25cm]{$\scriptstyle{p_1}$};
                \node(g0h0p1-tilde)[above of=z0-tilde,xshift=0.8cm,yshift=-0.25cm]{$\scriptstyle{g_0\widetilde{p}_1\widetilde{h}_0^{-1}}$};
                \path[draw=black] (z0) -- node{\firstmidarrow}(z0-tilde);
                \path[draw=black] (z1) -- node{\firstmidarrow}(z1-tilde);
                \path[draw=black] (z0-tilde) -- node{\firstmidarrow}(g0h0z1-tilde);
                \filldraw[black] (z0) circle (2pt) node[below] {$z_0$};
                \filldraw[black] (z1) circle (2pt) node[below] {$z_1$};
                \filldraw[black] (z2) circle (2pt) node[below] {$z_2$};
                \filldraw[black] (z0-tilde) circle (2pt) node[below] {$\widetilde{z}_0$};
                \filldraw[black] (z1-tilde) circle (2pt) node[below] {$\widetilde{z}_1$};
                \filldraw[black] (g0h0z1-tilde) circle (2pt) node[below,xshift=2mm] {$g_0\widetilde{z}_1\widetilde{h}_0^{-1}$};

                \node(lambda0)[above of=c,yshift=2cm]{};
                \node(lambda1)[above of=c,yshift=1cm]{};
                \path[->,draw=black] (lambda0) -- (lambda1);
                \node(lambda)[above of=c,yshift=1.5cm,xshift=3mm]{$\lambda$};
            \end{tikzpicture}
            \caption{$\Sigma$ respects composition.}\labell{f:sigma comp}
        \end{figure}       

        Now suppose we make different choices: for each $x\in X$ choose $z'\in\lambda^{-1}(x)$, leading to a right inverse $\Sigma'\colon\Pi(G\ltimes X)\to\Pi((G\times H)\ltimes Z)$ to $\Pi\lambda$.  By \cref{i:gamma conj} of \cref{l:gamma} (and \cref{r:gamma}), for each object $x_K$ of $\Pi(G\ltimes X)$ there is a unique arrow $$\left[\overline{(1_G,b)},z'\right]\colon \Sigma'(x_K)=z'_{\Gamma_{z'}}\to \Sigma(x_K)=z_{\Gamma_{z}}.$$ 

        Fix an arrow $[\widebar{g},p]\colon(x_0)_{K_0}\to(x_1)_{K_1}$ in $\Pi(G\ltimes X)$.  Let $\widetilde{p}$ be a lift of $p$ in $Z^{\Gamma_{z_0}}$ starting at $z_0$ and ending at $\widetilde{z}$.  Let $\widetilde{p}'$ be a lift of $p$ in $Z^{\Gamma_{z'}}$ starting at $z_0'$ and ending at $\widetilde{z}'$.  For $i=0,1$, let $b_i'\in H$ be the unique element such that $z_i'=z_i(b_i)^{-1}$, and let $\widetilde{h},\widetilde{h}'\in H$ be the unique elements such that $gz_1\widetilde{h}^{-1}=\widetilde{z}$ and $gz'_1(\widetilde{h}')^{-1}=\widetilde{z}'$, resp.\ (see \cref{f:choices}).

        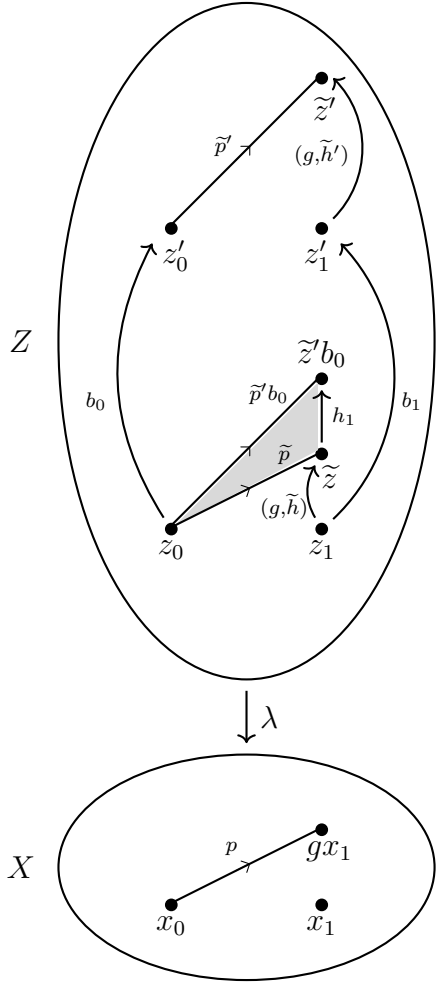
\begin{figure}
            \centering
            \begin{tikzpicture}[thick]
                \coordinate(d);
                \draw[black] (d) ellipse (2.5 and 1.5);
                \node(X)[left of=d,xshift=-2cm]{$X$};
                \node(x0)[left of=d,yshift=-0.5cm]{};
                \node(x1)[right of=d,yshift=-0.5cm]{};
                \node(gx1)[above of=x1]{};
                \node(p)[above of=x0,xshift=0.8cm,yshift=-0.25cm]{$\scriptstyle{p}$};
                \path[draw=black] ([xshift=-2mm]x0.east) -- node{\firstmidarrow} ([xshift=1mm]gx1.west);
                \filldraw[black] (x0) circle (2pt) node[below] {$x_0$};
                \filldraw[black] (x1) circle (2pt) node[below] {$x_1$};
                \filldraw[black] (gx1) circle (2pt) node[below,xshift=1mm] {$gx_1$};

                \draw[black] (d)[yshift=7cm] ellipse (2.5 and 4.5);
                \node(Z)[left of=d,xshift=-2cm,yshift=7cm]{$Z$};
                \node(z0)[left of=d,yshift=-0.5cm,yshift=5cm]{};
                \node(z1)[right of=d,yshift=-0.5cm,yshift=5cm]{};
                \node(z-tilde)[above of=z1]{};
                \filldraw[black] (z0) circle (2pt) node[below] {$z_0$};
                \filldraw[black] (z1) circle (2pt) node[below] {$z_1$};
                \filldraw[black] (z-tilde) circle (2pt) node[below,xshift=1mm] {$\widetilde{z}$};          

                \node(z'0)[left of=d,yshift=8.5cm]{};
                \node(z'1)[right of=d,yshift=8.5cm]{};
                \node(z'-tilde)[above of=z'1,yshift=1cm]{};
                \node(p'-tilde)[above of=z'0,xshift=0.7cm,yshift=1mm]{$\scriptstyle{\widetilde{p}'}$};
                \path[draw=black] ([xshift=-2mm]z'0.east) -- node{\secondmidarrow} ([xshift=1mm]z'-tilde.west);
                \filldraw[black] (z'0) circle (2pt) node[below] {$\;z'_0$};
                \filldraw[black] (z'1) circle (2pt) node[below] {$z'_1\;$};
                \filldraw[black] (z'-tilde) circle (2pt) node[below,xshift=1mm,yshift=-1mm] {$\widetilde{z}'\;$};  

                \path[->,draw=black] ([yshift=2mm]z0) to [bend left] ([xshift=-1mm,yshift=-2mm]z'0.west);
                \node(b0)[above of=z0,xshift=-1cm,yshift=0.7cm]{$\scriptstyle{b_0}$};
                \path[->,draw=black] ([yshift=2mm]z1) to [bend right=45] ([xshift=1mm,yshift=-2mm]z'1.east);
                \node(b1)[above of=z1,xshift=1.2cm,yshift=0.7cm]{$\scriptstyle{b_1}$};
                \node(z'-tildeb0)[above of=z-tilde]{};
                \filldraw[black] (z'-tildeb0) circle (2pt) node[above] {$\widetilde{z}'b_0$};
                \fill[gray!30] ([xshift=0.5mm,yshift=0.5mm]z0.center) to ([xshift=-0.5mm,yshift=0.5mm]z-tilde.center) -- ([xshift=-0.5mm,yshift=-0.5mm]z'-tildeb0.center);
                \path[draw=black] ([xshift=-2mm]z0.east) -- node{\firstmidarrow} ([xshift=1mm]z-tilde.west);
                \path[draw=black] ([xshift=-2mm]z0.east) -- node{\secondmidarrow} ([xshift=1mm]z'-tildeb0.west);
                \path[->,draw=black] (z-tilde) -- (z'-tildeb0);
                \node(h1)[above of=z1,yshift=0.5cm,xshift=3mm]{$\scriptstyle{h_1}$};
                \path[->,draw=black] (z1) to [bend left] (z-tilde);
                \node(gh-tilde)[above of=z1,yshift=-0.7cm,xshift=-0.5cm]{$\scriptstyle(g,\widetilde{h})$};
                \node(p-tilde)[above of=z0,xshift=1.5cm]{$\scriptstyle{\widetilde{p}}$};   
                \node(p'-tildeb0)[above of=p-tilde,xshift=-2mm,yshift=-2mm]{$\scriptstyle{\widetilde{p}'b_0}$};
                \path[->,draw=black] (z'1) to [bend right=45] (z'-tilde.east);
                \node(gh'-tilde)[above of=z'1]{$\scriptstyle(g,\widetilde{h}')$};

                \node(lambda0)[above of=c,yshift=1.5cm]{};
                \node(lambda1)[above of=c,yshift=0.5cm]{};
                \path[->,draw=black] (lambda0) -- (lambda1);
                \node(lambda)[above of=c,yshift=1cm,xshift=3mm]{$\lambda$};
            \end{tikzpicture}
            \caption{Different choices of $z_0$ and $z_1$ lead to different right inverses $\Sigma$, but they are all naturally isomorphic.}\labell{f:choices}
        \end{figure}

        Since $\widetilde{p}$ and $\widetilde{p}'b_0$ both start at $z_0$ and lift $p$, let $\widetilde{\Phi}$ be the homotopy $\rel\;\{0\}$ from $\widetilde{p}$ to $\widetilde{p}'b_0$ equal to the lift of the trivial homotopy from $p$ to itself via \cref{l:lifting}.  There is a path $h_s\in H$ for which $\widetilde{\Phi}(s,1)=\widetilde{p}(1)h_s^{-1}$. 
        In particular, $\widetilde{p}(1)h^{-1}_1=\widetilde{z}'b_0$, and by principality, $h_1=(b_0)^{-1}\widetilde{h}'b_1\widetilde{h}^{-1}$.  Thus, 
        $$                                     \left[\overline{\left(g,\widetilde{h}\right)},\widetilde{p}\right]=\left[\overline{\left(g,h_1\widetilde{h}\right)},\widetilde{p}'b_0\right]=\left[\overline{\left(g,(b_0)^{-1}\widetilde{h}'b_1\right)},\widetilde{p}'b_0\right].
        $$
        Expanding the term on the right-hand side and moving factors around, we get:
        $$\left[\overline{\left(g,\widetilde{h}'\right)},\widetilde{p}'\right]=\left[\overline{\left(1_G,b_0\right)},z'_0\right]\left[\overline{\left(g,\widetilde{h}\right)},\widetilde{p}\right]\left[\overline{\left(1_G,b_1\right)^{-1}},\widetilde{z}\right].$$

        We conclude that $x_K\mapsto\left(\left[\widebar{1_G},b\right]\colon z'_{\Gamma_{z'}}\to z_{\Gamma_z}\right)$ is a natural isomorphism $\Sigma'\Rightarrow\Sigma$, proving \cref{i:right inverse choices}.
    \end{proof}

    We are now ready to consider our main goal of this section:  coefficient systems and how they behave with respect to the constructions of \cref{s:act gpds}. 
    
    \begin{definition}[Coefficient System]\labell{d:coeff system}
        Let $X$ be a $G$-manifold.  A \textbf{coefficient system} on $X$ is a contravariant functor $\calA$ from $\Pi(G\ltimes X)$  to the category of abelian groups $\calA\colon\Pi(G\ltimes X)\to\Ab$.  
    \end{definition}

    \begin{proposition}[Pullback Coefficient System]\labell{p:pullback coeff system}
        Given a functor $F\colon\Pi(G\ltimes X)\to\Pi(H\ltimes Y)$ and a coefficient system $\calA$ on $\Pi(H\ltimes Y)$, there is a \textbf{pullback coefficient system} $F^*\calA$ on $\Pi(G\ltimes X)$ defined by $F^*\calA:=\calA\circ F$.  In particular, given an equivariant functor $\varphi\colon G\ltimes X\to H\ltimes Y$, there is a pullback coefficient system $\varphi^*\calA$ on $\Pi(G\ltimes X)$ defined by $\varphi^*\calA=\calA\circ\Pi\varphi$.
    \end{proposition}

    \begin{proof}
        The first statement is immediate, and the second follows from the fact that $\Pi\varphi$ is a covariant functor (\cref{c:pi functors}).
    \end{proof}

    \begin{definition}[Category of Coefficient Systems]\labell{d:coeff system categ}
        Given an action groupoid $G\ltimes X$, the corresponding \textbf{category of coefficient systems} is the functor category $\Coeff(G\ltimes X):=\Fun(\Pi(G\ltimes X)^{\text{op}},\Ab)$ whose objects are coefficient systems given by contravariant functors $\calA\colon\Pi(G\ltimes X)\to\Ab$ and whose arrows are natural transformations.
    \end{definition}

    It is now straightforward to check: 

    \begin{proposition}[$\Coeff$ is a Functor]\labell{p:coeff functor}
        The assignment $\Coeff\colon\ActGpd\to\Cat$ is a contravariant $2$-functor sending an action groupoid $G\ltimes X$ to $\Coeff(G\ltimes X)$; an equivariant functor $\varphi\colon G\ltimes X\to H\ltimes Y$ to $\Coeff(\varphi):=\circ\Pi\varphi\colon\Coeff(H\ltimes Y)\to\Coeff(G\ltimes X)$, defined as precomposition with $\Pi\varphi$; and a natural transformation $S\colon\varphi\Rightarrow\psi$ between equivariant functors $\varphi,\psi\colon G\ltimes X\to H\ltimes Y$ to the natural transformation $\Coeff(S):=*\Pi S^{-1}\colon\Coeff(\psi)\Rightarrow\Coeff(\varphi)$, defined to be horizontal precomposition by $\Pi S^{-1}$.
    \end{proposition}

  \mute{
    \begin{remark}\labell{r:coeff functor}
        In fact, $\Coeff$ can most likely be promoted to a pseudofunctor from action groupoids to bicategories, but possibly where the arrows of $\Coeff(G\ltimes X)$ are pseudo-natural transformations for a given $G\ltimes X$.  The main issue with this being strict is that smooth natural transformations of $\ActGpd$ are sent to \emph{pseudo}natural transformations in \cref{t:pi nat transf}.  This goes beyond our purposes, however, and so we leave this for future work.
    \end{remark}
    }

    We now wish to look at how coefficient systems can be moved across a bibundle.   By \cref{p:pullback coeff system}, we know that if we have a bibundle $Z$ between $X$ and $Y$, we can pull back a coefficient system on $Y$ to one on $Z$.   We now look at how we can also push it forward  to a coefficient system on $X$.  

    \begin{proposition}[Pushforward Coefficient System]\labell{p:pushforward coeff system}
        Given compact Lie groups $G$ and $H$, a bibundle $G\ltimes X\ot{\lambda} Z\uto{\rho} H\ltimes Y$, and a coefficient system $\calA$ on $\Pi((G\times H)\ltimes Z)$, there is a \textbf{pushforward coefficient system} $\lambda_*\calA$ on $\Pi(G\ltimes X)$ such that $\lambda^*\lambda_*\calA$ is naturally isomorphic to $\calA$.  Moreover, any two such pushforward coefficient systems are naturally isomorphic.
    \end{proposition}

    \begin{proof}
        Fix a choice of right inverse $\Sigma$ of $\Pi\lambda$, and define $\lambda_*\calA:=\Sigma^*\calA$.  Given a different choice $\Sigma'$ of right inverse, by \cref{p:right inverse}, there is a natural isomorphism $T\colon\Sigma\Rightarrow\Sigma'$ which induces is a natural isomorphism $\calA T\colon\Sigma^*\calA\Rightarrow(\Sigma')^*\calA$.  This proves the second statement.

        We need to show that $\lambda^*\lambda_*\calA$ is naturally isomorphic to $\calA$.  Fix a closed subgroup $L\leq G\times H$ and $z\in Z^L$.  Let $x:=\lambda(z)$ and $K:=\pr_1(L)$.  Then $x\in X^K$ and for any $(k,h)\in L$, $kz\zeta_z(k)^{-1}=z=kzh^{-1}$ where $\zeta_z$ is as in \cref{i:gamma} of \cref{l:gamma} and \cref{r:gamma}.  By principality, $h=\zeta_z(k)$, from which it follows that $L=\Gamma_z$.  Thus $$\lambda^*\lambda_*\calA(z_L)=\lambda^*\lambda_*\calA(z_{\Gamma_z})=\lambda_*\calA(x_K)=\calA(z'_{\Gamma_{z'}})$$ 
        for some $z'\in\lambda^{-1}(\lambda(z))$.  There is a unique $b\in H$ such that $z'=zb^{-1}$, and hence an invertible arrow $\left[\overline{(1_G,b^{-1})},z\right]\colon z_{\Gamma_{z}}\to z'_{\Gamma_{z'}}$ in $\Pi((G\times H)\ltimes Z)$, which gives an invertible homomorphism $$\calA\left(\left[\overline{(1_G,b^{-1})},z\right]\right)\colon\lambda^*\lambda_*\calA( z_{\Gamma_{z}})\to \calA(z_{\Gamma_{z}}).$$  We will show that these homomorphisms form a natural isomorphism. 

        To check naturality, we consider  a situation similar to that in \cref{f:choices}: suppose for $i=0,1$ we have objects $(z_i)_{L_i}$ in $\Pi((G\times H)\ltimes Z)$ with $x_i:=\lambda(z_i)\in X$, and $\left[\overline{(g,h)},\widetilde{p}\right]\colon(z_0)_{\Gamma_{z_0}}\to(z_1)_{\Gamma_{z_1}}$ is an arrow in $\Pi((G\times H)\ltimes Z)$ sent to $[\widebar{g},p]\colon(x_0)_{K_0}\to(x_1)_{K_1}$ by $\Pi\lambda$.   Then  $$\lambda^*\lambda_*\calA\left[\overline{(g,h)},\widetilde{p}\right]=\lambda_*\calA\left[\widebar{g},p\right]=\calA\left[\overline{(g,h')},\widetilde{p}'\right]$$ for some lift $\widetilde{p}'$ of $p$ starting at $z_0'\in\lambda^{-1}(x_0)$ and ending at $\widetilde{z}'=gz_1'(h')^{-1}$.  For $i=0,1$, by principality, there exist unique $b_i\in H$ such that $z_i'=z_i(b_i)^{-1}$.  Then $\widetilde{p}'b_0$ is a lift of $p$ starting at $z_0$ and ending at $\widetilde{z}'b_0$.  Let $\Phi$ be an homotopy $\rel\,\{0\}$ from $\widetilde{p}$ to $\widetilde{p}'b_0$, defined via  a lift  of the trivial homotopy sending $p$ to itself guaranteed by \cref{l:lifting}.  By principality there is a path $h_t$ in $H$ starting at $1_H$ such that $\Phi(t,1)=\widetilde{p}(1)h_t^{-1}$ for all $t\in[0,1]$.  Since $$\widetilde{z}'b_0=gz_1'(h')^{-1}b_0=gz_1(b_1)^{-1}(h')^{-1}b_0=\widetilde{p}(1)h(b_1)^{-1}(h')^{-1}b_0,$$ we have $$\widetilde{p}(1)h_1^{-1}=\widetilde{z}'b_0=\widetilde{p}(1)h(b_1)^{-1}(h')^{-1}b_0.$$  By principality, $$h_1=(b_0)^{-1}h'b_1h^{-1}.$$  It now follows that 
        \begin{align*}
        \left[\overline{(1_G,(b_0)^{-1})},z_0\right]\left[\overline{(g,h')},\widetilde{p}'\right] =&~ \left[\overline{(g,h_1h)},\widetilde{p}'b_0\right]\left[\overline{(1_G,(b_1)^{-1})},z_1\right] \\
        =&~ \left[\overline{(g,h)},\widetilde{p}\right]\left[\overline{(1_G,(b_1)^{-1})},z_1\right].
        \end{align*}
        We have the following commutative diagram:
        $$\xymatrix{
        \lambda^*\lambda_*\calA((z_0)_{L_0}) \ar@{=}[d] & &\lambda^*\lambda_*\calA((z_1)_{L_1}) \ar[ll]_{\lambda^*\lambda_*\calA\left[\overline{(g,h)},\widetilde{p}\right]} \ar@{=}[d] \\
        \calA((z'_0)_{\Gamma_{z'_0}}) \ar[d]_{\calA\left[\overline{(1_G,(b_0)^{-1})},z_0\right]} & & \calA((z'_1)_{\Gamma_{z'_1}}) \ar[d]^{\calA\left[\overline{(1_G,(b_1)^{-1})},z_1\right]} \ar[ll]_{\calA\left[\overline{(g,h')},\widetilde{p}'\right]} \\
        \calA((z_0)_{\Gamma_{z_0}}) & & \calA((z_1)_{\Gamma_{z_1}}) \ar[ll]^{\calA\left[\overline{(g,h)},\widetilde{p}\right]} \\
        }$$
        where, as above, $L_i=\Gamma_{z_i}$ for $i=0,1$.  This proves naturality.
    \end{proof}

   We now revisit our  examples from the previous sections to illustrate our results.  

    \begin{example}\labell{x:ab2}
        We revisit the equivalence of \cref{x:ab} between Groupoids A and B,  and the fundamental groupoids of \cref{x:a2} and \cref{x:b2}.  Consider the coefficient system on Groupoid A,  $\calA\colon\Pi(\ZZ/2\ltimes\SS^1)\to\Ab$ sending all objects to the trivial abelian group $0$ except for the $S_{\ZZ/2}$, which is sent to $\ZZ$, and all non-identity arrows are sent to trivial homomorphisms.  Note that this coefficient system cannot be formed in the setting of ordinary Bredon-Illman cohomology in the sense that it does not factor through the orbit category.

        The bibundle $P_\varphi$ of \cref{x:ab} has inverse bibundle 
        $$P_\varphi^{-1}=\left(\U(1)\ltimes (\U(1)\times_{\ZZ/2}\SS^1)\ot{\lambda}\U(1)\times\SS^1\overset{\simeq}{\uto{\rho}}\ZZ/2\ltimes \SS^1\right)$$  We can pull back $\calA$  along $\rho$ to $\rho^*\calA$ defined on $\Pi\left((\U(1)\times\ZZ/2)\ltimes(\U(1)\times\SS^1)\right)$, and then push forward along $\lambda$ to $\lambda_*\rho^*\calA$ defined on $\Pi\left(\U(1)\ltimes (\U(1)\times_{\ZZ/2}\SS^1)\right)$.   This results in the following coefficient system for Groupoid B:
        \begin{itemize}
            \item To any object $[e^{i\alpha},e^{i\beta}]_{\langle1\rangle}$ choose $(e^{i\alpha},e^{i\beta})$ in $Z$, which is assigned $\calA((e^{i\beta})_{\langle1\rangle})=0$.
            \item To any object $[e^{i\alpha},N]_{\langle-1\rangle}$ choose $(e^{i\alpha},N)$ in $Z$, which is assigned $\calA(N_{\ZZ/2})=0$.
            \item To any object $[e^{i\alpha},S]_{\langle-1\rangle}$ choose $(e^{i\alpha},S)$ in $Z$, which is assigned $\calA(S_{\ZZ/2})=\ZZ$.
            \item All non-identity arrows are sent to trivial homomorphisms.
        \end{itemize}
        By \cref{p:pushforward coeff system}, $\lambda^*\lambda_*\rho^*\calA$ is naturally isomorphic to $\rho^*\calA$.
    \end{example}

    \begin{example}\labell{x:ac2}
        We revisit the equivalence of \cref{x:ac} between Groupoids A and C, and the fundamental groupoids of \cref{x:a2} and \cref{x:c2}.  Starting from the coefficient system $\calA$ on Groupoid A of  \cref{x:ab2},  the pullback coefficient system $\psi^*\calA$ assigns $0$ to all objects of $\Pi(D_4\ltimes\SS^1)$ except for $N_{\ZZ/2\times 1}$ and $S_{\ZZ/2\times 1}$, which are sent to $\ZZ$, and all  non-identity arrows are sent to trivial homomorphisms except for those going between these two objects, which are sent to $\id_\ZZ$.  This and \cref{x:ab2} show how using the fundamental groupoid and twisted coefficients, instead of the orbit category as in ordinary Bredon-Illman cohomology, addresses the issues brought up in \cite[Example 5.3]{PS:translation}.
    \end{example}

    \begin{example}\labell{x:d4}
        Consider Groupoid D from \cref{x:d1} and its fundamental groupoid of \cref{x:d3}. Let $\calR_n$ be the coefficient system on $\Pi(\SO(n)\ltimes\SS^n)$ defined by sending an object $x_H$ to $\Rep(H)$, the representation ring of the subgroup $H$, and an arrow $[\widebar{g},x]\colon x_H\to (g^{-1}\cdot x)_{g^{-1}Hg}$ to precomposition of the representation by conjugation by $g^{-1}$, and an arrow $[\widebar{1},p]\colon x_H\to y_K$ (whose existence implies $H\leq K$) to the restriction of the $K$-representation $\calR_n(y_K)$ to the corresponding $H$-representation $\calR_n(x_H)$.  Since representation rings can distinguish between the groups $\SO(n)$, it follows that $\calR_n$ and $\calR_m$ are equivalent if and only if $n=m$.  
    \end{example}

\section{Twisted Bredon-Illman Cohomology}\labell{s:bredon-illman}

    In this section, we define twisted Bredon-Illman cohomology and prove our long-promised main result, that the resulting theory is Morita-invariant in \cref{t:bredon-illman}.   The setup is similar to that of singular cohomology, but with an equivariant twist: the coefficient systems are defined on the fundamental groupoid.  We use similar language and notation as that in \cite[Section 3]{MuMu} to define it; note in \cite{MuMu} that the cohomology is referred to as Bredon-Illman cohomology with local coefficients. 

    \begin{definition}[Equivariant Singular Simplices]\labell{d:simplex}
        Let $\Delta_n=\langle d_0,\dots, d_n\rangle$ be the standard $n$-simplex with vertices $d_i$, and let $f_j\colon\Delta_{n-1}\to\Delta_n$ be the $j$th face operator, sending $d_i$ to $d_i$ for $i<j$ and to $d_{i+1}$ for $i\geq j$.
        
        Given a $G$-manifold $X$, an \textbf{equivariant (singular) $n$-simplex of $X$} is a $G$-equivariant map $\sigma\colon\Delta_n\times G/H\to X$ for some closed subgroup $H\leq G$.

        Given a face map $f_j$ of $\Delta_n$ and an equivariant $n$-simplex $\sigma\colon\Delta_n\times G/H\to X$, let  $\sigma^{(j)}$ denote the $(n-1)$-simplex $\Delta_{n-1}\times G/H\to X$ given by the composition $\sigma\circ(f_j\times\id_{G/H})$, called the \textbf{$j$th face of $\sigma$}.
\end{definition}

We can use evaluation on basepoints to project these simplices into the fundamental groupoid $\Pi(G \ltimes X)$.  

\begin{definition}[Face Maps and the Fundamental Groupoid]\labell{d:face maps fund gpd}
    An equivariant $n$-simplex $\sigma\colon\Delta_n\times G/H\to X$ induces a corresponding object of $\Pi(G\ltimes X)$ given by evaluation on the basepoint, $\sigma_H:=\sigma(d_0,H)_H\colon G/H\to X$.   
    
    We can also create an arrow of $\Pi(G \ltimes X)$, $[\sigma^{(j)}_H]\colon \sigma_H \to \sigma^{(j)}_H$ corresponding to a face map $f_j$ in the following way.    
    For $j>0$, the $j$th face map preserves the basepoint $d_0$.    Therefore we define the arrow $[\sigma^{(j)}_H] $ to be the identity  arrow of $\Pi(G \ltimes X)$ defined by 
    $$[\sigma^{(j)}_H]:=[\widebar{1_G},\sigma(d_0,H)]\colon\sigma_H\to \sigma^{(j)}_H$$ where $\sigma(d_0,H)$ denotes the constant path at $\sigma(d_0,H)$. 

    For $j=0$, the face map deletes the basepoint and so we need to create a non-constant path from the basepoint of $\sigma$ to the basepoint of $\sigma^{(0)}$.   The basepoint of $\sigma^{(0)}$ is $d_1$ in the original simplex, so $\sigma^{(0)}_H = \sigma(d_1, H)_H $ and we define  $[\sigma^{(0)}_H]$  to be the arrow of $\Pi(G\ltimes X)$ given by $$[\sigma^{(0)}_H]:=[\widebar{1_G},p]\colon\sigma_H \to \sigma^{(0)}_H$$ where $p$ is the interpolation $[0,1]\to X\colon t\mapsto\sigma((1-t)d_0+td_1,H)$.
\end{definition}

We recall that the fundamental groupoid $\Pi(G \ltimes X)$ allows us to twist continuously within orbits.    We want to define the same sort of twisting for simplices. 

    \begin{definition}[Orbit Twist Maps and Compatibility]\labell{d:orbit twist}
         An {\bf orbit twist map} is a continuous equivariant map $a\colon\Delta_n\times G/H\to\Delta_n\times G/K$ that  preserves the coordinate $u$ of $\Delta_n$.  Explicitly,    $\pr_1\circ a(u,gH)=u$ for all $u\in\Delta_n$.  Given two equivariant $n$-simplices $\sigma\colon \Delta_n\times G/H\to X$ and $\tau\colon\Delta_n\times G/K\to X$ and an orbit twist map $a$, we say  $\sigma$ and $\tau$ are \textbf{$a$-compatible} if $\sigma=\tau\circ a.$
    \end{definition}
         
An orbit twist map  $a$ induces a family of $G$-equivariant maps indexed by $u \in \Delta_n$,  $\widebar{a}_u\colon G/H\to G/K$ given by $\widebar{a}_u(gH)=\pr_2\circ a(u,gH)$.  Then if we set $\widebar{a}_0:=\widebar{a}_{d_0}$, if $\sigma, \tau$ are $a$-compatible then  $\sigma_H=\tau_K\circ\widebar{a}_0$ and $a$ induces an arrow of $\Pi(G\ltimes X)$ with a constant path  $[\widebar{a}_0,\sigma(d_0,H)]:  \sigma_H \to \tau_K$. 

    \begin{definition}[Twisted Bredon-Illman Cohomology]\labell{d:cochains}
        Given a $G$-manifold $X$ and a coefficient system $\calA$, let $C^n_\BI(G\ltimes X;\calA)$ be the group of all functions $c$ defined on equivariant $n$-simplices taking values  $c(\sigma)\in\calA(\sigma_H)$ for any  equivariant $n$-simplex $\sigma\colon\Delta_n\times G/H\to X$.  Let $S^n_\BI(G\ltimes X;\calA)$ be the subgroup of $C^n_\BI(G\ltimes X;\calA)$ such that for any orbit twist map $a$ and $a$-compatible equivariant $n$-simplices $\sigma\colon\Delta_n\times G/H\to X$ and $\tau\colon\Delta_n\times G/K\to X$,  $$c(\sigma)=\calA([\widebar{a}_0,\sigma(d_0,H)])c(\tau);$$ this is the \textbf{group of equivariant (singular) $n$-cochains of $X$}.

        Let $\delta\colon S^n_\BI(G\ltimes X; A)\to S^{n+1}_\BI(G\ltimes X;\calA)$ be the \textbf{boundary homomorphism}, given by $$\delta c(\sigma):=\sum_{j=0}^{n+1}(-1)^j  \calA\left([\sigma_H^{(j)}]\right) c(\sigma^{(j)}).$$  Recall that $[\sigma^{(j)}]$ is defined to be the identity map $[\widebar{1_G},\sigma(d_0,H)]$ for $j>0$, and so this formula simplifies to $$\delta c(\sigma):=\calA\left([\sigma_H^{(0)}]\right)c(\sigma^{(0)})+\sum_{j=1}^{n+1}(-1)^jc(\sigma^{(j)}).$$
        Then $\delta^2=0$, resulting in the \textbf{twisted Bredon-Illman cochain complex} $(S^\bullet_\BI(G\ltimes X;\calA),\delta)$ and associated \textbf{twisted Bredon-Illman cohomology groups} $H^\bullet_\BI(G\ltimes X;\calA)$.  See \cite{illman,MuMu} for details.
    \end{definition}

    \begin{remark}\labell{r:cochains}
        Let $G\ltimes X$ be a $G$-space, $H\leq G$ a closed subgroup, $x\in X$, and $g\in G$.  Observe that that the arrow $[\widebar{g},x]\colon x_H\to (g^{-1}x)_{g^{-1}Hg}$ induces an orbit twist map $$\widecheck{g}:=\id_{\Delta_k}\times\widebar{g}\colon\Delta_k\times G/H\to\Delta_k\times G/(g^{-1}Hg)\colon (u,g'H)\mapsto(u,g'gg^{-1}Hg)=(u,g'Hg).$$  Thus, given a $k$-simplex $\sigma\colon\Delta_k\times G/H\to X$, there is a corresponding $k$-simplex $\sigma^g\colon\Delta_k\times G/(g^{-1}Hg)\to X$ that is $\widecheck{g}$-compatible with $\sigma$: $$\sigma=\sigma^g\circ\widecheck{g}.$$  Given a coefficient system $\calA$ and $c\in S^k_\BI(G\ltimes X;\calA)$ we have $$c(\sigma)=\calA([\widebar{g},x])c(\sigma^g).$$
        Thus $c$ respects the action of $G$ on $X$ (and is equivariant in the sense that it also respects the induced action of $G$ on the image of $\calA$ in $\Ab$). Moreover, since $g$ is invertible, the value $c(\sigma^g)$ is determined by $c(\sigma)$.
    \end{remark}  
    
    As remarked in \cite[Page 204]{MuMu}, if each fixed point set $X^K$ is path-connected and the coefficient system is simple (\emph{i.e.} independent of paths), then one obtains the ordinary Bredon-Illman cohomology.  See \cite[Definition 2.5, Proposition 2.6]{MuMu} for details.
   
Next we consider functoriality of our cohomology theory; see \cite[Proposition 3.9]{MuMu} for a similar statement. 

    \begin{proposition}[$H_\BI^\bullet$ is Functorial]\labell{p:pullback BI}
        Given an equivariant functor $\varphi\colon G\ltimes X\to H\ltimes Y$ and a coefficient system $\calA$ on $\Pi(H\ltimes Y)$, there is a pullback homomorphism $\varphi^*\colon H_\BI^\bullet(H\ltimes Y,\calA)\to H_\BI^\bullet(G\ltimes X,\varphi^*\calA)$. 
    \end{proposition}

    To prove this, we need to show how equivariant simplices interact with equivariant functors.  

    \begin{lemma}\labell{l:flat simplices}
        Let $\varphi\colon G\ltimes X\to H\ltimes Y$ be an equivariant functor, and let $\sigma$ be an equivariant $n$-simplex of $G\ltimes X$.  There is an equivariant $n$-simplex $\varphi_\flat\sigma$ of $H\ltimes Y$ making the following diagram commute
        $$\xymatrix{
            \Delta_n\times G/K \ar[r]^{\quad\quad\sigma} \ar[d]_{\id_{\Delta_n}\times\widetilde{\varphi}_K} & X \ar[d]^{\varphi} \\
            \Delta_n\times H/\widetilde{\varphi}(K) \ar[r]_{\quad\quad\varphi_\flat\sigma} & Y \\
        }$$
        where $\widetilde{\varphi}_K\colon G/K\to H/\widetilde{\varphi}(K)$ is the induced equivariant map.
    \end{lemma}

    \begin{proof}
        The homomorphism $\widetilde{\varphi}\colon G\to H$ induces a ($G$-$H$)-equivariant map $\widetilde{\varphi}_K\colon G/K\to H/\widetilde{\varphi}(K)$ sending $gK$ to $\widetilde{\varphi}(g)\widetilde{\varphi}(K)$.  Since $\sigma$ is completely determined by its restriction to $\Delta_n\times\{K\}$, we may define $$\varphi_\flat\sigma(u,h\widetilde{\varphi}(K)):=h\varphi(\sigma(u,K)).$$  The diagram in the statement of the lemma now commutes.
    \end{proof}

    We now prove \cref{p:pullback BI}.

    \begin{proof}[Proof of \cref{p:pullback BI}]
        Fix $c\in S_\BI^n(H\ltimes Y,\calA)$.  Define $\varphi^*c$ on equivariant $n$-simplices of $G\ltimes X$ by $$\varphi^*c(\sigma\colon\Delta_n\times G/K\to X):=c(\varphi_\flat \sigma\colon\Delta_n\times H/\widetilde{\varphi}(K)\to Y).$$   By construction, $\varphi^*$ is linear. Since $$\varphi^*c(\sigma)\in\calA((\varphi_\flat\sigma)_{\widetilde{\varphi}(K)})=\varphi^*\calA(\sigma_K),$$ $\varphi^*c$ is a well-defined element of $C_\BI^n(G\ltimes X;\varphi^*\calA)$.

        We first check that $\phi^*c$ is an equivariant $n$-cochain.   Suppose we have  another equivariant $n$-simplex of $X$, $\sigma'\colon\Delta_n\times G/K'\to X$  that is $a$-compatible to $\sigma$ for an orbit twist map $a$,  so $\sigma=\sigma'\circ a$.  Define $\varphi_\flat a\colon\Delta_n\times H/\widetilde{\varphi}(K)\to\Delta_n\times G/\widetilde{\varphi}(K')$ by $$\varphi_\flat a(u,h\widetilde{\varphi}(K)):=(u,ha_u\widetilde{\varphi}(K')).$$  The following diagram commutes:
        $$\xymatrix{
            \Delta_n\times G/K \ar[rr]^{a} \ar[ddd]_{\id_{\Delta_n}\times\widetilde{\varphi}_K} \ar[dr]_{\sigma} & & \Delta_n\times G/K' \ar[ddd]^{\id_{\Delta_n}\times\widetilde{\varphi}_{K'}} \ar[dl]^{\sigma'} \\
             & X \ar[d]^{\varphi} & \\
             & Y & \\
            \Delta_n\times H/\widetilde{\varphi}(K) \ar[ur]^{\varphi_\flat\sigma} \ar[rr]_{\varphi_\flat a} & & \Delta_n\times H/\widetilde{\varphi}(K'). \ar[ul]_{\varphi_\flat\sigma'} \\
        }$$
        In particular, $\varphi_\flat a$ preserves $\Delta_n$-coordinates and so $\varphi_\flat a$ is an orbit twist map such that   $\varphi_\flat\sigma=\varphi_\flat\sigma'\circ\varphi_\flat a$. Therefore $\varphi_\flat \sigma$ and $\varphi_\flat \sigma'$ are $\varphi_\flat a$-compatible.  

        Then we check:   $$\varphi^*c(\sigma)=c(\varphi_\flat\sigma)=c(\varphi_\flat\sigma'\circ\varphi_\flat a)=\calA([\overline{\widetilde{\varphi}(a_0)},\varphi(\sigma(d_0,K))])c(\varphi_\flat\sigma')=\varphi^*\calA([\widebar{a}_0,\sigma(d_0,K)])\varphi^*c(\sigma').$$
        Thus $\varphi^*c\in S_\BI^n(G\ltimes X;\varphi^*\calA)$.
       
        We now check that $\varphi^*$ respects the boundary homomorphism $\delta$.  Let $\sigma$ be an equivariant $(n+1)$-simplex of $X$.  It follows from the following commutative diagram
        $$\xymatrix{
            \Delta_{n-1}\times G/K \ar[d]_{\id_{\Delta_{n-1}}\times\widetilde{\varphi}_K} \ar[rr]_{\quad f_j\times\id_{G/K}} \ar@/^2pc/[rrr]^{\sigma^{(j)}} & & \Delta_n\times G/K \ar[d]_{\id_{\Delta_n}\times\widetilde{\varphi}_K} \ar[r]_{\quad\quad\sigma} & X \ar[d]^{\varphi} \\
            \Delta_{n-1}\times H/\widetilde{\varphi}(K) \ar[rr]_{f_j\times\id_{H/\widetilde{\varphi}(K)}} \ar@/_2pc/[rrr]_{\varphi_\flat\sigma^{(j)}} & & \Delta_n\times H/\widetilde{\varphi}(K) \ar[r]^{\quad\quad\varphi_\flat\sigma} & Y \\
        }$$
        that $(\varphi_\flat\sigma)^{(j)}=\varphi_\flat(\sigma^{(j)})=:\varphi_\flat\sigma^{(j)}.$  Also, $[\sigma^{(0)}_K]:=[\widebar{1_G},p]$ where $p$ is the interpolation from $t\mapsto \sigma((1-t)d_0+td_1,K)$, and this is sent by $\Pi\varphi$ to $[\overline{\widetilde{\varphi}(1_G)},\varphi\circ p]=[\widebar{1_H},\varphi\circ p]=[(\varphi_\flat\sigma^{(0)})_{\widetilde{\varphi}(K)}]$.
        Then
        \begin{align*}
            \varphi^*\delta c(\sigma)   =&~ \delta c(\varphi_\flat\sigma) \\
                =&~ \calA\left(\left[(\varphi_\flat\sigma^{(0)})_{\widetilde{\varphi}(K)}\right]\right)c(\varphi_\flat\sigma^{(0)})+\sum_{j=1}^{n+1}(-1)^jc(\varphi_\flat\sigma^{(j)}) \\
                =&~ \calA\left(\Pi\varphi[\sigma^{(0)}_K]\right)c(\varphi_\flat\sigma^{(0)})+\sum_{j=1}^{n+1}(-1)^jc(\varphi_\flat\sigma^{(j)}) \\
                =&~ \varphi^*\calA([\sigma^{(0)}_K])\varphi^*c(\sigma^{(0)})+\sum_{j=1}^{n+1}(-1)^j\varphi^*c(\sigma^{(j)}) \\
                =&~ \delta\varphi^*c(\sigma).
        \end{align*}
        This completes the proof.
    \end{proof}

Thus we have verified that our cohomology theory gives a contravariant functor.   We also observe the following functoriality with respect to coefficient systems. 

    \begin{proposition}[Functoriality of $H_\BI^\bullet$ in Coefficient Systems]\labell{p:nat isom cohom}
        For any $G$-space $X$ and coefficient systems $\calA$ and $\calB$, a natural transformation $\eta\colon\calA\Rightarrow\calB$ induces a map of cohomology groups $\eta_*\colon H^\bullet_\BI(G\ltimes X;\calA)\to H^\bullet_\BI(G\ltimes X;\calB)$.  This makes $H^\bullet_\BI(G\ltimes X;\cdot)$ into a functor from $\Coeff(G\ltimes X)$ to $\Ab$.  In particular, if $\eta$ is a natural isomorphism, then $\eta_*$ is an isomorphism of cohomology groups.
    \end{proposition}

    \begin{proof}
        Suppose $\eta\colon\calA\Rightarrow\calB$ is a natural isomorphism.  Define $\eta_*\colon S_\BI^n(G\ltimes X;\calA)\to C_\BI^n(G\ltimes X;\calB)$ as follows: for $\sigma\colon\Delta_n\times G/K\to X$ an equivariant $n$-simplex, let $$\eta_*(c)(\sigma):=\eta_{\sigma_K}(c(\sigma)).$$  It follows from the naturality of $\eta$ that $\eta_*c\in S_\BI^n(G\ltimes X;\calB)$. 

        Fix $\sigma\colon\Delta_{n+1}\times G/K\to X$ an equivariant $(n+1)$-simplex, and $c\in S_\BI^n(G\ltimes X;\calA)$.
        \begin{align*}
            \eta_*(\delta c)(\sigma)  =&~ \eta_*\left(\calA([\sigma_K^{(0)}])c(\sigma^{(0)})+\sum_{j=1}^{n+1}(-1)^jc(\sigma^{(j)})\right) \\
                =&~ \calB([\sigma_K^{(0)}])\eta_*c(\sigma^{(0)})+\sum_{j=1}^{n+1}(-1)^j\eta_*c(\sigma^{(j)}) \\
                =&~ \delta(\eta_*c)(\sigma)
        \end{align*}
        It follows that $\eta$ induces a homomorphism $\eta_*\colon H_\BI^\bullet(G\ltimes X;\calA)\to H_\BI^\bullet(G\ltimes X;\calB).$  

        The assignment $\calA\to H^\bullet_\BI(G\ltimes X,\calA)$ and $(\eta\colon\calA\to\calB)\mapsto\eta_*$ is functorial; in particular, if $\eta$ is a natural isomorphism, then $\eta_*$ is an isomorphism (in fact, on the level of cochains).  The result follows.
    \end{proof}

It follows from \cref{p:pullback BI} that $G\ltimes X\mapsto H^\bullet_\BI(G\ltimes X,\cdot)$ is natural: $$H^\bullet_\BI(G\ltimes X,\cdot)\circ\Coeff(\varphi)=\varphi^*\circ H^\bullet_\BI(H\ltimes Y,\cdot)$$ for any equivariant functor $\varphi\colon G\ltimes X\to H\ltimes Y$.

Now we start considering Morita equivalence  and examining what happens to twisted Bredon-Illman cohomology when we have a bibundle morphism.  
    
    \begin{proposition}[$H^\bullet_\BI$ and Pushforwards]\labell{p:pushforward BI}
        Given compact Lie groups $G$ and $H$, a bibundle $G\ltimes X\ot{\lambda} Z\uto{\rho} H\ltimes Y$, and a coefficient system $\calB$ of $\Pi(G\ltimes X)$, there is an isomorphism $$\lambda^*\colon H_\BI^\bullet(G\ltimes X;\calB)\to H_\BI^\bullet((G\times H)\ltimes Z;\lambda^*\calB).$$
    \end{proposition}

    To prove \cref{p:pushforward BI}, we will need to lift equivariant $n$-simplices of $X$ to $Z$.  The following lemma spells out how these lifts behave. 

    \begin{lemma}\labell{l:lift simplices}
         Given compact Lie groups $G$ and $H$, a bibundle $G\ltimes X\ot{\lambda} Z\uto{\rho} H\ltimes Y$, and a coefficient system $\calB$ of $\Pi(G\ltimes X)$,
        \begin{enumerate}
            \item\labell{i:lift}Let $K\leq G$ be a closed subgroup, and let $\sigma\colon\Delta_n\times G/K\to X$ be an equivariant $n$-simplex. There exists a \textbf{lift of $\sigma$} to $Z$; that is, a subgroup $\Gamma \leq G \times H$ and an  equivariant $n$-simplex $\tau\colon\Delta_n\times(G\times H)/\Gamma\to Z$ such that $$\lambda\circ\tau=\sigma\circ(\id_{\Delta_n}\times(\pr_1)_K)$$ where the map $(\pr_1)_K\colon (G\times H)/\Gamma\to G/K$ is the ($(G\times H)$-$G$)-equivariant map induced by $\pr_1\colon G\times H\to G$.
            
            \item\labell{i:ind of lift}If $\tau'$ is another lift of $\sigma$, then for any  $c\in S^n_\BI((G\times H)\ltimes Z,\lambda^*\calB)$, $$c(\tau)=c(\tau').$$
            
            \item\labell{i:lift a}If $a\colon\Delta_n\times G/K\to \Delta_n\times G/K'$ is an orbit twist map and $\sigma'\colon\Delta_n\times G/K'$ is an equivariant $n$-simplex which is $a$-compatible with $\sigma$, then for any lift $\tau$ of $\sigma$ and lift $\tau'$ of $\sigma'$, there exists an orbit twist map $\widehat{a}\colon (G\times H)/\Gamma\to(G\times H)/\Gamma'$ such that $\tau$ and $\tau'$ are $\widehat{a}$-compatible,   
        \end{enumerate}
    \end{lemma}

    \begin{proof}  
        Fix $z\in\lambda^{-1}(\sigma(d_0,K))$ and define   $\Gamma=\Gamma_z$ to be the subgroup fixing $z$ and projecting to $K$ as in \cref{i:gamma} of \cref{l:gamma}.   The equivariant $n$-simplex $\sigma$ is completely determined by the ordinary $n$-simplex $$\sigma|_{\Delta_n\times\{K\}}\colon\Delta_n\to X\colon u\mapsto \sigma(u,K).$$  By the homotopy lifting property, since $\lambda\colon Z\to X$ is a principal $H$-bundle, there is a lift $\widehat{\sigma}\colon\Delta_n\to Z$ of $\sigma|_{\Delta_n\times\{K\}}$ to $Z$ such that $\widehat{\sigma}(d_0)=z$.  Define the equivariant $n$-simplex $\tau\colon\Delta_n\times(G\times H)/\Gamma_z\to Z$ by $\tau(u,(g,h)\Gamma_z):=(g,h)\widehat{\sigma}(u)$.

        The homomorphism $\pr_1\colon G\times H\to G$ induces an equivariant map $(\pr_1)_K\colon (G\times H)/\Gamma_z\to G/K$ given by $(\pr_1)_K((g,h)\Gamma_z):=gK.$  Then $$\lambda\circ\tau(u,(g,h)\Gamma_z)=\lambda((g,h)\widehat{\sigma}(u))=g\sigma(u,K)=\sigma(u,(\pr_1)_K((g,h)\Gamma_z)),$$ proving the first statement.

        For the second statement, fix $c\in S_\BI^n((G\times H)\ltimes Z;\lambda^*\calB)$. Let $z'=\tau'(d_0,\Gamma_{z'})$, so that  $\tau'\colon\Delta_n\times(G\times H)/\Gamma_{z'}\to Z$.   Since $\lambda\colon Z\to X$ is a principal $H$-bundle, the map $\Delta_n\to H$ sending $u$ to the unique $b_u\in H$ such that $\tau(u,\Gamma_z)(b_u)^{-1}=\tau'(u,\Gamma_{z'})$ is continuous.  Moreover, letting $a\colon\Delta_n\times(G\times H)/\Gamma_{z'}\to\Delta_n\times(G\times H)/\Gamma_z$ be the continuous map $(u,(g,h)\Gamma_{z'})\mapsto(u,(g,hb_u)\Gamma_z)$, we have $\tau'=\tau\circ a$, and so $\tau$ and $\tau'$ are $a$-compatible.  Thus, $$c(\tau')=\lambda^*\calB([\widebar{a},z'])c(\tau)=\calB([\widebar{1_G},\sigma(d_0,K)])c(\tau)=c(\tau).$$

        For the last statement, suppose $\sigma\colon\Delta_n\times G/K\to X$ and $\sigma'\colon\Delta_n\times G/K'\to X$ are equivariant $n$-simplices and $a$ is an orbit twist map such that $\sigma$ and $\sigma'$ are $a$-compatible.  Fix $z\in\lambda^{-1}(\sigma(d_0,K))$, and let $\Gamma_z$ and $\Gamma'_z$ be the associated subgroups of $G\times H$ as in \cref{i:gamma} of \cref{l:gamma} corresponding to $K$ and $K'$, resp.     Let $\tau$ be a lift of $\sigma$ to $Z$ such that $\tau(d_0,\Gamma_z)=z$, and let $\tau'$ be a lift of $\sigma'$ such that $\tau'(d_0,\Gamma'_z)=z$. We will show that there exists an orbit twist map $\widehat{a}\colon (G\times H)/\Gamma_z\to(G\times H)/\Gamma'_z$ such that $\tau$ and $\tau'$ are $\widehat{a}$-compatible, where $\lambda(z)=\sigma(d_0,K)$, $(\pr_1)_K(\Gamma_z)=K$, and $(\pr_1)_{K'}(\Gamma'_z)=K'$.
        
        Since $\sigma=\sigma'\circ a$, it follows that for each $u\in U$ there exists a unique $b_u\in H$ such that $$\tau(u,(g,h)\Gamma_z)(b_u)^{-1}=\tau'(u,(g,h)a_u\Gamma'_z)$$ where $a_u\in G$ induces $\overline{a}_u$.  It follows from principality that $b_u$ is continuous in $u$.  Define $$\widehat{a}\colon\Delta_n\times(G\times H)/\Gamma_z\to\Delta_n\times(G\times H)/\Gamma'_z\colon(u,(g,h)\Gamma_z)\mapsto (u,(ga_u,h(b_u)^{-1})\Gamma'_z).$$  Then $\widebar{a}_u\circ(\pr_1)_K=(\pr_1)_{K'}\circ\widebar{\widehat{a}}_u$ for each $u\in\Delta_n$.  We have the following commutative diagram, completing the proof.
        $$\begin{gathered}[b]\xymatrix{
            \Delta_n\times(G\times H)/\Gamma_z \ar[ddd]_{\id_{\Delta_n}\times(\pr_1)_K} \ar[rr]^{\widehat{a}} \ar[dr]_{\tau} & & \Delta_n\times(G\times H)/\Gamma'_z \ar[dl]^{\tau'} \ar[ddd]^{\id_{\Delta_n}\times(\pr_1)_{K'}} \\
             & Z \ar[d]^{\lambda} & \\
             & X & \\
            \Delta_n\times G/K \ar[rr]_{a} \ar[ur]^{\sigma} & & \Delta_n\times G/K' \ar[ul]_{\sigma'} \\
        }\\[-\dp\strutbox]\end{gathered}\eqno\qedhere$$
    \end{proof}

    We are now ready to prove \cref{p:pushforward BI}.

    \begin{proof}[Proof of \cref{p:pushforward BI}]
        By \cref{p:pullback BI}, $$\lambda^*\colon H_\BI^\bullet(G\ltimes X;\calB)\to H_\BI^\bullet((G\times H)\ltimes Z;\lambda^*\calB)$$ exists, and so we only need to show that it is a bijection.

        Suppose $c\in S^n_\BI(G\ltimes X;\calB)$ satisfies $\lambda^*c=0$.  For any equivariant $n$-simplex $\tau$ of $Z$, $$c(\lambda_\flat\tau)=0.$$  But for any equivariant $n$-simplex $\sigma$ of $X$, $c(\sigma)=c(\lambda_\flat\tau)=\lambda^*c(\tau)=0$ where $\tau$ is a lift of $\sigma$, which exists by \cref{i:lift} of \cref{l:lift simplices}; moreover, this is independent of the lift $\tau$ of $\sigma$ by \cref{i:ind of lift} of \cref{l:lift simplices}.  Thus $c=0$, and $\lambda^*\colon S_\BI^\bullet(G\ltimes X;\calB)\to S_\BI^\bullet((G\times H)\ltimes Z;\lambda^*\calB)$ is injective.

        Suppose $\widehat{c}\in S_\BI^n((G\times H)\ltimes Z;\lambda^*\calB)$.  Define for any equivariant $n$-simplex $\sigma\colon\Delta_n\times G/K\to X$ the function $c(\sigma):=\widehat{c}(\tau)$ where $\tau$ is any lift of $\sigma$ to $Z$.  By \cref{l:lift simplices}, $\tau$ exists and $c(\sigma)$ is independent of the lift $\tau$; thus $c$ is a well-defined element of $C_\BI^n(G\ltimes X;\calB)$.  Suppose $\sigma'\colon\Delta_n\times G/K'\to X$ is another equivariant $n$-simplex of $X$ and $a\colon\Delta_n\times G/K\to\Delta_n\times G/K'$ is an orbit twist map such that $\sigma$ and $\sigma'$ are $a$-compatible. Let $\tau$, $\tau'$, and $\widehat{a}$ be as in \cref{i:lift a} of \cref{l:lift simplices}.  Then $$c(\sigma)=\widehat{c}(\tau)=\widehat{c}(\tau'\circ\widehat{a})=\lambda^*\calB\left(\left[\overline{\widehat{a}}_0,z\right]\right)\widehat{c}(\tau')=\calB\left(\left[\overline{a}_0,\sigma(d_0,K)\right]\right)c(\sigma').$$  Thus $c\in S_\BI^n(G\ltimes X,\calB)$.  It follows that $\lambda^*$ is surjective.
    \end{proof}

    We now prove the Main Theorem of the paper, reworded in terms of the language developed above.

    \begin{theorem}[Morita Invariance of Twisted Bredon-Illman Cohomology]\labell{t:bredon-illman} 
        If $G\ltimes X$ and $H\ltimes Y$ are Morita equivalent Lie group action groupoids with $G$ and $H$ compact, and $\calA$ is a coefficient system on the discrete tom Dieck fundamental groupoid $\Pi(H\ltimes Y)$, then there is an isomorphism between the twisted Bredon-Illman cohomologies $H^\bullet_\BI(G\ltimes X;\lambda_*\rho^*\calA)\cong H^\bullet_\BI(H\ltimes Y;\calA)$ for any biprincipal bibundle $G\ltimes X\ot{\lambda}Z\underset{\raisebox{1mm}{$\scriptstyle{\rho}$}}{\subdwe}H\ltimes Y$.
    \end{theorem}

    \begin{proof}
        Let $G\ltimes X\ot{\lambda}Z\underset{\raisebox{1mm}{$\scriptstyle{\rho}$}}{\subdwe}H\ltimes Y$ be a biprincipal bibundle.  By \cref{p:pushforward BI} applied to the equivariant functor $\rho$, there is an isomorphism $$\rho^*\colon H_\BI^\bullet(H\ltimes Y;\calA)\to H_\BI^\bullet((G\times H)\ltimes Z;\rho^*\calA)$$  Then we consider the pushforward coefficient system $\lambda_*\rho^*A$ on $X$, and apply \cref{p:pushforward BI} to the functor $\lambda$ to get   an isomorphism $$\lambda^*\colon H_\BI^\bullet(G\ltimes X;\lambda_*\rho^*\calA)\to H_\BI^\bullet((G\times H)\ltimes Z;\lambda^*\lambda_*\rho^*\calA).$$  By \cref{p:pushforward coeff system}, there is a natural isomorphism $\eta\colon\lambda^*\lambda_*\rho^*\calA\Rightarrow\rho^*\calA$.  By \cref{p:nat isom cohom}, this induces an isomorphism $$\eta_*\colon H_\BI^\bullet((G\times H)\ltimes Z;\lambda^*\lambda_*\rho^*\calA)\to H_\BI^\bullet((G\times H)\ltimes Z;\rho^*\calA).$$  Composing, we get an isomorphism $$(\rho^*)^{-1}\circ\eta_*\circ\lambda^*\colon H_\BI^\bullet(G\ltimes X;\lambda_*\rho^*\calA)\to H_\BI^\bullet(H\ltimes Y;\calA).\eqno\qedhere$$
    \end{proof}

    A note on functoriality:  We can express the functoriality developed throughout this paper by using a Grothendieck construction of the functor $\Coeff$ from action groupoids to coefficient systems, themselves functors from fundamental groupoids of action groupoids  to $\Ab$ (see \cite[Chapter 10]{JY} for details of this construction): 

    \begin{corollary}\labell{c:bredon-illman}
        Let $\int_{\ActGpd}\Coeff$ be the Grothendieck category defined by the functor $\Coeff$.   This category has objects defined as pairs $(G\ltimes X,\calA)$, where $G\ltimes X$ is an action groupoid of a compact Lie group and $\calA$ is a coefficient system on $\Pi(G\ltimes X)$, and arrows from $(G\ltimes X,\calA)$ to $(H\ltimes Y,\calB)$ defined as pairs $$(\varphi\colon G\ltimes X\to H\ltimes Y,\,\eta\colon \varphi^*\calB\Rightarrow \calA)$$ where $\varphi$ is an equivariant functor and $\eta$ is a natural transformation.  Then $H^\bullet_\BI$ is a functor from $\int_\ActGpd\Coeff$ to $\Ab$, sending an object $(G\ltimes X,\calA)$ to $H_\BI^\bullet(G\ltimes X,\calA)$ and an arrow $(\varphi\colon G\ltimes X\to H\ltimes Y,\eta\colon \varphi^*\calB\Rightarrow \calA)$ to the homomorphism $\eta_*\circ\varphi^*$. 
    \end{corollary}

    Our results also suggest some type of functoriality on the localized category $\ActGpd[\calW_\eq^{-1}]$ of action groupoids defined up to Morita equivalence, with  morphisms defined by bibundles (see discussion after \cref{c:bibundle form}).    However, our current results do not immediately give us this, since given a bibundle $P:=\left(G\ltimes X\underset{\lambda}{\longleftarrow}Z\underset{\rho}{\longrightarrow}H\ltimes Y\right)$ the construction of the right inverse functor 
    $\Sigma$ to $\Pi\lambda\colon\Pi((G\times H)\ltimes Z)\to\Pi(G\ltimes X)$ requires a choice  and different choices $\Sigma_1,\Sigma_2$ differ by a natural isomorphism $\eta\colon\Sigma_1\Rightarrow\Sigma_2$; see \cref{p:right inverse}.   It seems plausible that a higher categorical perspective might allow this functoriality to be expressed, but we do not attempt to do so here.

    Before moving onto our earlier examples and computing their cohomologies, we prove a couple of applications of \cref{t:bredon-illman}.  For the first, we retrieve a known result, which is a consequence of the fact that Morita equivalent Lie groupoids have homeomorphic orbit spaces \cite[Theorem 4.3.1]{dH:orbispaces} (in fact, diffeomorphic \cite[Theorem 3.8]{watts:lgpd-derham}).

    \begin{corollary}\labell{c:cohom orbit sp}
        Given Morita equivalent action groupoids $G\ltimes X$ and $H\ltimes Y$, the orbit spaces $X/G$ and $Y/G$ have isomorphic singular cohomologies.
    \end{corollary}

    \begin{proof}
        Let $G\ltimes X\ot{\lambda}Z\underset{\raisebox{1mm}{$\scriptstyle\rho$}}{\subdwe} H\ltimes Y$ be a biprincipal bibundle.  Let $\calA$ (resp.\ $\calB$) be the coefficient system on $\Pi(G\ltimes X)$ (resp.\ $\Pi(H\ltimes Y)$) sending each object to $\ZZ$ and each arrow to the identity map.  Then $\calA=\lambda_*\rho^*\calB$, with the resulting twisted Bredon-Illman cohomology isomorphic to the singular cohomology of the orbit space $X/G$ (resp.\ $Y/H$).
    \end{proof}

    For our second application, we define a coefficient system similar to that in \cref{x:d4} for general action groupoids, and show that Morita equivalent action groupoids admit isomorphic twisted Bredon-Illman cohomologies with values from these coefficient systems.  We then define a particular interesting family of cocycles which are Morita invariants.

    Given an action groupoid $G\ltimes X$, let $\calR_X$ be the coefficient system on $\Pi(G\ltimes X)$ defined by sending each object $x_K$ to the representation ring $\Rep(K)$ for each closed subgroup $K\leq G$, each arrow of the form $[\widebar{1_G},p]$ to the identity homomorphism for each path $p$, each arrow induced by an inclusion $K\hookrightarrow K'$ to the restriction operator, and each arrow of the form $[\widebar{g},x]$ to precomposition with conjugation.  

    For each $x\in X$, by the Slice Theorem, there is a linear action of $\Stab(x)$ on the normal space $V_x:=T_xX/T_x(G\cdot x)$, known as the isotropy or normal representation \cite[Theorems 2.3.3, 2.4.1]{DK}.  This is a Morita invariant \cite[Theorem 4.3.1]{dH:orbispaces}.  (In fact, one can see this directly using Lashof's perspective of equivariant bundles \cite{lashof}.)  More precisely, given a biprincipal bibundle $G\ltimes X\ot{\lambda}Z\underset{\raisebox{1mm}{$\scriptstyle\rho$}}{\subdwe} H\ltimes Y$, for a fixed $x\in X$ and a choice of $z\in\lambda^{-1}(x)$, $\lambda$ induces an isomorphism between the isotropy representations on $V_z$ and $V_x$, and since $\rho$ is $G$-principal, it similarly induces an isomorphism between the isotropy representations on $V_z$ and $V_{\rho(z)}$.

    Define an equivariant $k$-cochain $c^k_{G\ltimes X}\in C_\BI^k(G\ltimes X;\calR_X)$ that sends each $k$-simplex $\sigma\colon\Delta_k\times G/H\to X$ to the restriction of the isotropy representation to $H\leq\Stab(x)$ where $x=\sigma(d_0,H)$.

    \begin{corollary}\labell{c:rep coeff sys}
        Let $G\ltimes X$ and $H\ltimes Y$ be Morita equivalent action groupoids.  Then $H^\bullet_\BI(G\ltimes X;\calR_X)$ is isomorphic to $H^\bullet_\BI(H\ltimes Y;\calR_Y)$, and this isomorphism sends a cocycle $c^k_{G\ltimes X}$ to a similarly-defined cocycle $c^k_{H\ltimes Y}$ for each $k$.
    \end{corollary}

    \begin{proof}
        Let $G\ltimes X\ot{\lambda}Z\underset{\raisebox{1mm}{$\scriptstyle\rho$}}{\subdwe} H\ltimes Y$ be a biprincipal bibundle.  Fix a choice of $z\in\lambda^{-1}(x)$ for each $x\in X$. Given an arrow $[\widebar{g},p]\colon x_K\to (x')_{K'}$, choices $z\in\lambda^{-1}(x)$ and $z'\in\lambda^{-1}(x')$, a lift $\widetilde{p}$ of $p$ to $Z^{\Gamma_z^K}$ starting at $z$ and ending at $gz'h^{-1}$ for some unique $h\in H$, we have the following diagram of Lie groups:
        \begin{equation}\labell{e:morita invts}
            \begin{gathered}
            \xymatrix{
                K \ar[rr]^{\gamma_K} \ar[d]_{C_{g^{-1}}} && \Gamma_z^K \ar[d]_{C_{(g^{-1},h^{-1})}} \ar[r]^{\widetilde{\rho}} & \widetilde{\rho}(\Gamma^K_z) \ar[d]^{C_{h^{-1}}} \\
                g^{-1}Kg \ar[rr]_{\gamma_{g^{-1}Kg}} \ar@{^{(}->}[d] && \Gamma_{z'}^{g^{-1}Kg} \ar@{^{(}->}[d] \ar[r]^{\widetilde{\rho}} & \widetilde{\rho}(\Gamma^{g^{-1}Kg}_{z'}) \ar@{^{(}->}[d] \\
                K' \ar[rr]_{\gamma_{K'}} && \Gamma_{z'}^{K'} \ar[r]_{\widetilde{\rho}} & \widetilde{\rho}(\Gamma_{z'}^{K'}),\\
            }
            \end{gathered}
        \end{equation}
        where $\gamma_K\colon k\mapsto (k,\zeta_z(k))$ (similarly for $\gamma_{g^{-1}Kg}$ and $\gamma_{K'}$), $C_{g^{-1}}$ is conjugation by $g^{-1}$ (similarly for the other such notated maps), and recall $\widetilde{\rho}=\pr_2$.  

        It follows from $H$-principality of $\lambda$ that $\zeta_z(k)=h\zeta_{z'}(g^{-1}kg)h^{-1}$, from which it follows that the top-left square of \cref{e:morita invts} commutes.  The bottom-left square commutes by \cref{l:gamma}.  The remaining two squares commute by the naturality of projection maps.  It is immediate that each of the left horizontal maps are isomorphisms, and that each restriction of $\widetilde{\rho}$ on the right is an isomorphism onto its image follows from the fact that $\rho$ is $G$-principal by \cref{c:bibundle form}.  It follows that $\calR_X$ and $\lambda_*\rho^*\calR_Y$ are naturally isomorphic, from which it follows from \cref{p:nat isom cohom} that $H^\bullet_\BI(G\ltimes X;\calR_X)\cong H^\bullet_\BI(G\ltimes X;\lambda_*\rho^*\calR_Y)$.  This, in turn is naturally isomorphic to $H^\bullet_\BI(H\ltimes Y,\calR_Y)$ by \cref{t:bredon-illman}, proving the first statement.

        It is straightforward to check that $c^k_{G\ltimes X}\in S_\BI^k(G\ltimes X;\calR_X)$.  It follows from the discussion above that the isomorphism $H_\BI^k(\calG\ltimes X;\calR_X)\to H_\BI^k(\calH\ltimes Y;\calR_Y)$ sends $c^k_{G\ltimes X}$ to a similarly-defined $c^k_{H\ltimes Y}$ in the case of cocycles.  This proves the second statement.
    \end{proof}

    We end the paper by computing the twisted Bredon-Illman cohomology for our examples.
    
    \begin{example}[{\bf Groupoid A}]\labell{x:a3}
        We continue with the groupoid $\ZZ/2 \ltimes \SS^1$ from \cref{x:a1} and the coefficient system which is $\zed$ on the south pole $S$ and zero elsewhere from \cref{x:ab2}.  We compute $H_\BI^0(\ZZ/2\ltimes\SS^1;\calA)$ and $H_\BI^1(\ZZ/2\ltimes\SS^1;\calA)$.  The remaining cohomology groups are trivial, which follows from \cite[Theorem 7.3]{MuMu}. 
        The equivariant $0$-simplices are:
        \begin{enumerate}
            \item for each $x\in\SS^1$, the simplex $\Delta_0\times(\ZZ/2)/\langle1\rangle\to\SS^1$, with image equal to $\{x\}$; and
            \item the two simplices $\Delta_0\times(\ZZ/2)/(\ZZ/2)\to\SS^1$ with image $\{N\}$ or $\{S\}$.
        \end{enumerate}

        The equivariant $1$-simplices are:
        \begin{enumerate}
            \item orbits of paths $\Delta_1\times(\ZZ/2)/\langle1\rangle\to\SS^1$; and
            \item constant maps $\Delta_1\times(\ZZ/2)/(\ZZ/2)\to\SS^1$ to $N$ or $S$.
        \end{enumerate}

        Suppose $c\in S^0_\BI(\ZZ/2\ltimes\SS^1;\calA)$.  Since the only non-zero values of $\calA$ are at $S$,  $c$ vanishes on the $0$-simplices except for the simplex $\Delta_0\times(\ZZ/2)/(\ZZ/2)\to\SS^1$ with image in $S$;  this simplex gets sent to $\calA(S) = \ZZ$.    Since this integer can be arbitrary, we have $S^0_\BI(\ZZ/2\ltimes\SS^1;\calA)\cong\ZZ$.

        Computing the boundary of $c$ above, for a $1$-simplex $\sigma\colon\Delta_1\times(\ZZ/2)/K\to\SS^1$,
        $$\delta c(\sigma)=\calA\left([\widebar{1},\sigma((1-t)d_0+td_1,K)]\right)c(\sigma^{(0)})-c(\sigma^{(1)}).$$
        which is trivial unless $K=\ZZ/2$ and $\sigma$ takes image in $\{S\}$.  If $\sigma$ is a constant map to $S$, then $\sigma^{(0)}=\sigma^{(1)}$ and  the arrow $[\overline{1},\sigma((1-t)d_0+td_1,K)]$ will have a constant map and so be a unit, and so $\calA\left([\widebar{1},\sigma((1-t)d_0+td_1,K)]\right)$ is the identity map.   So $\delta c(\sigma) = 0$.   It follows that $H_\BI^0(\ZZ/2\ltimes\SS^1;\calA)\cong\ZZ$.

        Similarly, given $c\in S^1_\BI(\ZZ/2\ltimes\SS^1;\calA)$, the only simplex on which $c$ does not vanish is $\Delta_1\times(\ZZ/2)/(\ZZ/2)\to\SS^1$ with constant image $S$.  This also takes a value as an arbitrary integer, and so $S_\BI^1(\ZZ/2\ltimes\SS^1;\calA)\cong\ZZ$.

        When we compute the boundary of $c$ above, for a $2$-simplex $\tau\colon\Delta_2\times(\ZZ/2)/K\to\SS^1$,
        $$\delta c(\tau)=\calA\left([\widebar{1},\tau((1-t)d_0+td_1,K)]\right)c(\tau^{(0)})-c(\tau^{(1)})+c(\tau^{(2)}).$$
        If $K=\langle1\rangle$, or if $K=\ZZ/2$ and the image of $\tau$ is $\{N\}$, then this is trivial.  However, if $K=\ZZ/2$ and the image of $\tau$ is $\{S\}$, then $\tau^{(0)}=\tau^{(1)}=\tau^{(2)}$ and $[\widebar{1},\tau((1-t)d_0+td_1,\ZZ/2)]$ is sent to the identity homomorphism as above.  Thus  $\delta c(\tau)=c(\tau^{(0)})$ and the only way to have $\delta c =0$ is for $c = 0$.   It follows that $H^1_\BI(\ZZ/2\ltimes\SS^1;\calA)=0$.
    \end{example}
 
    \begin{example}[{\bf Groupoid B}] \labell{x:b3}
        Turning to $\U(1)\ltimes Y$ where $Y:=\U(1)\times_{\ZZ/2}\SS^1$, the equivariant $0$-simplices are:
        \begin{enumerate}
            \item for each $y\in Y$, the simplex $\Delta_0\times\U(1)/\langle1\rangle\to Y$, with image equal to $\{y\}$;
            \item simplices $\Delta_0\times\U(1)/\langle-1\rangle$ with image in one of the cross-caps.
        \end{enumerate}

        The equivariant $1$-simplices are:
        \begin{enumerate}
            \item orbits of paths $\Delta_1\times\U(1)/\langle1\rangle\to Y$; and
            \item orbits of paths $\Delta_1\times\U(1)/\langle-1\rangle\to Y$ with image contained in one of the cross-caps.
        \end{enumerate}

        Suppose $c\in S^0_\BI(\U(1)\ltimes Y;\lambda_*\rho^*\calA)$.  Then $c$ vanishes on all $0$-simplices except for those of the form $\Delta_0\times\U(1)/\langle-1\rangle\to Y$ with image in the bottom cross-cap, which must take value in $\calA(S_{\ZZ/2})=\ZZ$.  A similar computation to the above yields that $H^0_\BI(\U(1)\ltimes Y;\lambda_*\rho^*\calA)\cong\ZZ$.

        For $c\in S^1_\BI(\U(1)\ltimes Y;\lambda_*\rho^*\calA)$, again, $c$ vanishes on all but the $1$-simplices $\Delta_1\times\U(1)/\langle-1\rangle\to Y$ with image in the bottom cross-cap, which has value in $\ZZ$.  A computation similar to the above yields that $c$ is a cocycle if and only if $c(\tau^{(0)})=c(\tau^{(1)})-c(\tau^{(2)})$ for any $2$-simplex $\tau\colon\Delta_2\times\U(1)/\langle-1\rangle\to Y$ with image in the bottom cross-cap.  Here we see our compatibility condition come into play:  since this cross-cap is an entire orbit, there is an orbit twist map  $a\colon\Delta_1\times\U(1)/\langle-1\rangle$ collapsing any $1$-cell to a ``constant'' $1$-cell (it is constant on $\Delta_1\times\langle-1\rangle$).  Thus, $c$ is determined by its values on these constant $1$-cells.  For any constant $1$-cell $\sigma$, let $\tau$ be the constant $2$-cell with the same value at $(u,\langle-1\rangle)$.  The cocycle conditions then requires then that $c(\sigma)=2c(\sigma)$ since each face of $\tau$ is a copy of $\sigma$.  Therefore, $c=0$.  It follows that $H^1_\BI(\U(1)\ltimes Y;\lambda_*\rho^*\calA)=0$.  Similar to \cref{x:a3}, the higher cohomology groups are trivial.  Since $\ZZ/2\ltimes\SS^1$ and $\U(1)\ltimes Y$ are Morita equivalent, we have verified \cref{t:bredon-illman}.
    \end{example}    

    \begin{example}[{\bf Groupoid C}]\labell{x:c3}
        A computation similar to that of \cref{x:b3} also yields that $H^0_\BI(D_4\ltimes\SS^1;\psi^*\calA)\cong\ZZ$ and $H^1_\BI(D_4\ltimes\SS^1;\psi^*\calA)=0$ where $D_4\ltimes\SS^1$ is as in \cref{x:c1}. This is the same as the cohomology determined in \cref{x:b3}, which again verifies \cref{t:bredon-illman}, as $\psi$ is an equivariant weak equivalence (hence a Morita equivalence).  Since $\psi^*\calA$ is only non-zero on   the $\{N,S\}$ orbit, this can be interpreted as yielding the cohomology of the orbit space of $\{N,S\}$ in $D_4\ltimes\SS^1$, which is the same as that of $\{S\}$ in $(\ZZ/2)\ltimes\SS^1$.  This computation is not possible with ordinary Bredon-Illman cohomology:  in the ordinary setting, the coefficient system on  $(\ZZ/2)\ltimes\SS^1$ would have to assign the same value to both $N$ and $S$, and we would not get an analogue coefficient system living on this groupoid.   By defining the coefficient system on the fundamental groupoid instead, we obtain the desired result.
    \end{example}

    For our last example, we show how \cref{c:rep coeff sys} fails for similar action groupoids that are not Morita equivalent.

    \begin{example}[{\bf Groupoid D}]\labell{x:d5}
        Continuing \cref{x:d4}, we compute $H^0_\BI(\SO(n)\ltimes\SS^n;\calR^n)$.  
               
        Fix $c\in S^0_\BI(\SO(n)\ltimes\SS^n;\calR_n)$.
        The two poles $\{N,S\}$ of $\SS^n$ are the only $\SO(n)$-fixed points.  There is an arrow $[\widebar{1},N]$ induced by the inclusion $g\SO(n-1)g^{-1}\hookrightarrow\SO(n)$ for each $g$.  Each of these correspond to orbit twist maps between $0$-simplices $\sigma^g_N\colon(d_0,g\SO(n-1)g^{-1})\mapsto N$ and $\sigma_N\colon(d_0,\SO(n))\mapsto N$.  Thus, since $c$ respects compatibility, its value at $\sigma_N^g$ must be a restriction of its value at $\sigma_N$ for all $g$.   $\SO(n)$ is the union of all of these conjugate subgroups $g\SO(n-1)g^{-1}$ as $g$ runs through $\SO(n)$, and thus the value of $c$ at $\sigma_N$ is determined by its values on the $0$-simplices $\sigma^g_N$, which are in turn determined by $\sigma^1_N$ by \cref{r:cochains}.  A similar argument holds if $N$ is replaced with $S$.  Also, for any $x\in\SS^n$ that is fixed by a closed subgroup $H\leq\SO(n)$, we have $H\leq g\SO(n-1)g^{-1}$ for some $g$.  It follows from the fact that $c$ respects compatibility that the $0$-simplex $\Delta_0\times G/H\to \SS^n\colon (d_0,H)\mapsto x$ is sent by $c$ to the restriction of the representation of $$c(\Delta_0\times G/(g\SO(n-1)g^{-1})\to\SS^n\colon (d_0,g\SO(n-1)g^{-1})\mapsto x).$$ Thus, again applying \cref{r:cochains}, we see that it is sufficient to consider the values of $c$ on $0$-simplices of the form $\Delta_0\times\SO(n)/\SO(n-1)\to\SS^n$.

        Suppose now that $c$ is a cocycle, so $\delta c=0$.  For any $1$-simplex $\tau\colon\Delta_1\times\SO(n)/K\to\SS^n$, the cocycle condition reduces to $$\calR_n\!\left([\tau^{(0)}_K]\right)c(\tau^{(0)})=c(\tau^{(1)}).$$
        Suppose $\tau\colon\Delta_1\times\SO(n)/\SO(n-1)\to\SS^n$ is a $1$-simplex such that $u\mapsto\tau(u,\SO(n-1))$ is a path $p$ in the great circle $(\SS^n)^{\SO(n-1)}$.  Since $\calR_n[\widebar{1},p]$ is the identity map, the cocycle condition guarantees that the $0$-simplices corresponding to the endpoints of $p$ are sent to the same value by $c$; in particular, $c$ is constant along all $0$-simplices $\Delta_0\times\SO/\SO(n-1)\to\SS^n$ sending $(d_0,\SO(n-1))$ to the great circle $(\SS^n)^{\SO(n-1)}$.  Thus, a cocycle $c$ is completely determined by its value on the single $0$-simplex $\sigma^1_N\colon\Delta_0\times\SO/\SO(n-1)\to\SS^n$.
      
       Since $\sigma^1_N$ can be sent to an arbitrary representation of $\SO(n-1)$ by $c$, it follows that $H^0_\BI(\SO(n)\ltimes\SS^n;\calR_n)\cong\Rep(\SO(n-1))$.  In particular, $H^\bullet_\BI(\SO(n)\ltimes\SS^n;\calR_n)\cong H^\bullet_\BI(\SO(m)\ltimes\SS^m;\calR_m)$ if and only if $n=m$.  Recalling that the orbit spaces of these actions $\SO(n)\ltimes\SS^n$ are all diffeomorphic (see the footnote of \cref{x:d2}), this is an example in which twisted Bredon-Illman cohomology can differentiate between different stabilizers, which are subtler Morita invariants.
    \end{example}


\printbibliography

\end{document}